\documentclass[12pt]{article}

\usepackage{amsmath,amsthm,amsfonts,amssymb,amscd,cite}
\usepackage{graphicx}

\allowdisplaybreaks[4]

\newtheorem {lemma}{Lemma}
\newtheorem {theorem} {Theorem}
\newtheorem {corollary}{Corollary}
\newtheorem {proposition}{Proposition}
\newtheorem{definition}{Definition}

\usepackage{fullpage}
\usepackage{tikz}

\usepackage{listings}
\begin{document}

\title{On the multiplicity of matching polynomial roots and
$\theta$-critical graphs}

\author{Leyou Xu\footnote{School of Mathematical Sciences, South China Normal University,
Guangzhou 510631, P.R. China. Email: leyouxu@m.scnu.edu.cn.} }

\date{}
\maketitle

\begin{abstract}
The matching polynomial of a graph encodes rich combinatorial information through its roots.
We determine the maximum multiplicity of a non-zero matching polynomial root and characterize all graphs attaining the bound. We also generalize the result to any fixed $\theta$, where the graphs attaining the bound are related to $\theta$-critical graphs. Inspired by these graphs, we give a constructive answer to Godsil's question. 
Finally, we show the existence of $1$-critical tree of order $n$ for all $n\ge 9$ and $1$-critical graph of order $n$ for all $n\ge 5$, and describe a method to construct $1$-critical graphs from existing ones. 

\end{abstract}

\section{Introduction}

We consider simple graphs in this paper, i.e., graphs without loops and multiedges. For a graph $G$, $V(G)$ and $E(G)$ denote the vertex set and edge set of $G$, respectively. If it is not mentioned elsewhere, we assume that $|V(G)|=n$.
A matching of a graph $G$ is a set of pairwise non-adajcent edges of $G$. Evidently, the maximum matching of $G$ contains at most $\lfloor\frac{n}{2}\rfloor$ edges. For any non-negative integer $k$, let $p_G(k)$ be the number of matching with $k$ edges in $G$. It's clear that $p_G(0)=1$ and $p_G(1)=|E(G)|$. The matching polynomial of $G$ is defined as 
\[
\mu(G,x)=\sum_{k=0}^{\lfloor n/2\rfloor}(-1)^kp_G(k)x^{n-2k}.
\]
The matching polynomial is related to the characteristic polynomial $\phi(G,x)$, and $\mu(G,x)=\phi(G,x)$ if and only if $G$ is a forest \cite[Corollary 4.2]{GG}. Matching polynomial is exactly the average of the characteristic polynomial over all signs of the edge set \cite{GG1}. A remarkable result on bipartite Ramanujan graphs \cite{WSS} was proved using this fact, along with the upper bound on the absolute value of matching polynomial roots. Building on this idea, we focus instead on the upper bound on the multiplicity of matching polynomial roots.

A graph $G$ is said to have a perfect matching if it has a matching covering each vertex of $G$, i.e., it has a matching with $\frac{n}{2}$ edges, which reqiures $n$ to be even. This shows that $0$ is not a root of $\mu(G,x)$. Let $m(\theta,G)$ be the multiplicity of $\theta$ as a root of $\mu(G,x)$. For $\theta=0$, $m(\theta,G)$ is known as the deficiency of $G$, which is the number of vertices missed by some maximum matching. Many well-known results, such as Gallai's lemma and the Edmonds-Gallai structure theorem, are related to $m(0,G)$. Ku and Chen \cite{KC} gave a generalization of the Edmonds–Gallai structure theorem for an arbitrary root $\theta$. To our knowledge, there are very few results concerning $m(\theta,G)$ when $\theta\ne 0$. Motivated by this, we investigate $m(\theta,G)$ for non-zero $\theta$ in this paper. 

A well-known result tells the maximum multiplicity of a root of $\mu(G,x)$ is at most the minimum number of vertex disjoint paths needed to cover the vertex set
of $G$, see \cite[Theorem 4.5]{Godsil} for detailed proof. This leads that if $G$ has a Hamiltonian path, then all roots of $\mu(G,x)$ are simple. Ku and Wong \cite{KW1,KW2} gave a necessary and sufficient condition for the maximum multiplicity of a root of $\mu(G,x)$ to be equal to the minimum number of vertex disjoint paths needed to cover it.
On the other hand, it's nature to consider the maximum multiplicity of a given root $\theta$ of $\mu(G,x)$. It's easy to see that $m(0,G)\le n-2$ with equality if and only if $G$ is a star. So the first goal of this work is to determine the maximum multiplicity of a non-zero matching polynomial root. We also characterize all graphs attaining the bound, in which the set $\mathcal{F}_n$ is defined in Section 3. The first main result is stated as follows.

\begin{theorem}\label{bound}
Let $G$ be a connected graph of order $n\ge 7$.
If $\theta\ne 0$, then $m(\theta,G)\le \lfloor\frac{n-3}{2}\rfloor$ with equality if and only if $\theta=\pm 1$ and $G\in \mathcal{F}_{n}$.
\end{theorem}

As is shown in Theorem \ref{bound}, the upper bound is attained only if $\theta=\pm 1$. However, the maximum multiplicity of $\theta$ as a matching polynomial root remains unclear for fixed non-zero $\theta$ with $\theta\ne \pm 1$. Here, we turn our attention to the case for general non-zero $\theta$.
The second goal of this paper is to determine the maximum multiplicity of any given non-zero root $\theta$ as a matching polynomial root. The corresponding result is stated as follows. The notations $n_\theta$ and $\mathcal{H}_{\theta}^n$ are defined in Section 4, in which graphs are related to $\theta$-critical graphs. It's evident that if $G$ is $\theta$-critical, then $m(\theta,G)=1$. So we only need to consider graphs that is not $\theta$-critical.

\begin{theorem}\label{essbound}
Let $G$ be a connected graph that is not $\theta$-critical. Then $m(\theta,G)\le \frac{n-n_\theta-1}{n_\theta}$ with equality if and only if $n\equiv 1\pmod {n_\theta}$ and $G\in \mathcal{H}_{\theta}^n$.
\end{theorem}

Inspired by the graphs attaining the bound above, we give a constructive answer to Godsil's question \cite[Section 6]{Godsil1}, which asked whether a $\theta$-positive vertex must be $\theta$-special when $\theta\ne 0$. We give a constructive answer regardless of whether $\theta=0$ or not.

\begin{theorem}\label{spec}
A $\theta$-positive vertex does not need to be $\theta$-special.
\end{theorem}

A $0$-critical graph is known as the factor critical graph, which is a graph with an odd number of vertices in which deleting one vertex in every possible way results in a graph with a perfect matching. $0$-critical graphs have been well studied in the literature, almost nothing is known about $\theta$-critical graphs for $\theta\ne 0$. We aim to explore the structure of $\theta$-critical graphs in general or for some specific $\theta$. Godsil stated in \cite{Godsil1} that it might be interesting to investigate the case $\theta=1$ in depth. This gives us a motivation to investigate $1$-critical graphs. It's well-known that $0$-critical graphs have odd order. But for $1$-critical graphs, it does not depend on whether the order is even or odd.

\begin{theorem}\label{1}
For any $n\ge 9$, there exists a $1$-critical tree of order $n$. Moreover, for any $n\ge 5$, there exists a $1$-critical graph of order $n$ that is not tree.
\end{theorem}

The outline of this paper is as follows. In Section 2, we show some basic and useful properties of matching polynomials. In Section 3, we prove Theorem \ref{bound}. In Section 4, we focus on properties of $\theta$-critical graphs and graphs containing non-zero $\theta$ as a matching polynomial root that are not $\theta$-critical, respectively, and as a consequence, prove Theorems \ref{essbound} and \ref{spec}. In Section 5, we prove Theorem \ref{1} and give a way to obtained a $1$-critical graph from an existing one.

\section{Basic properties}

We introduce some terminology and basic results on matching polynomial, which are useful for our proofs. For undefined notation and terminology we refer to \cite{BM}.

For a vertex $u\in V(G)$, $G\setminus u$ denotes the graph obtained from $G$ by deleting the vertex $u$ and its incident edges. 
For a path $P$ of $G$, $G\setminus P$ denotes the graph obtained from $G$ by deleting vertices of $P$ and all edges incident to vertices of $P$. Particularlly, if $P=uv$, then $G\setminus uv=(G\setminus u)\setminus v$. 
For an edge $e\in E(G)$, $G-e$ denotes the graph obtained from $G$ by deleting the edge $e$. 
For two disjoint graphs $G$ and $H$, $G\cup H$ denotes the disjoint union of these two graphs.

\begin{proposition}\label{rule}\cite{Godsil}
The matching polynomial satisfies the following identities:\\
(i) $\mu(G\cup H,x)=\mu(G,x)\mu(H,x)$;\\
(ii) $\mu(G,x)=\mu(G-e,x)-\mu(G\setminus uv,x)$ if $e=uv\in E(G)$;\\
(iii) $\mu(G,x)=x\mu(G\setminus u)-\sum_{vu\in E(G)}\mu(G\setminus uv,x)$.
\end{proposition}

The following result can be obtained directly from the definition of matching polynomial.

\begin{proposition}\label{symmetric}
The roots of $\mu(G,x)$ are symmetric about the origin.
\end{proposition}

Let $\mathcal{P}(u)$ denote the set of all paths starting at $u$ in $G$. The path tree of $G$ relative to $u$, denoted by $T(G,u)$, is a tree with vertes set $\mathcal{P}(u)$, and two paths are adjacent if and only if one is a maximum proper subpath of the other one. The following famous result plays an important role.

\begin{theorem}\label{pathtree}\cite{Godsil}
Let $u$ be a vertex of graph $G$ and $T=T(G,u)$ be the path tree with respect to $u$. Then \[
\frac{\mu(G\setminus u,x)}{\mu(G,x)}=\frac{\mu(T\setminus u,x)}{\mu(T,x)}.
\]
Moreover, if $G$ is connected, then $\mu(G,x)$ divides $\mu(T,x)$. 
\end{theorem}

As $\mu(T,x)=\phi(T,x)$ for any forest $T$, Theorem \ref{pathtree} implies that all roots of $\mu(G,x)$ are real. From Theorem \ref{pathtree} and Perron-Frobenius theorem, the following result follows immediately.

\begin{lemma}\label{simple}\cite{Godsil}
If $G$ is connected, then the largest root of $\mu(G,x)$ is simple, and is strictly larger than the largest root of $\mu(G\setminus u,x)$. 
\end{lemma}

Moreover, the roots of $\mu(G\setminus u, x)$ interlace those of $\mu(G,x)$, more precisely, the following result holds (see also \cite[Corollary 1.3]{Godsil}).

\begin{lemma}\label{interlacing}
For $u\in V(G)$, $m(\theta,G)-1\le m(\theta,G\setminus u)\le m(\theta,G)+ 1$.
\end{lemma}

As a consequence of Lemma \ref{interlacing}, for any real number $\theta$, we can classify the vertices of a graph based on whether the multiplicity of $\theta$ increases by one, decreases by one, or remains unchanged when the vertex is deleted as follows.

\begin{definition}
Let $\theta$ be a root of $\mu(G,x)$. For $u\in V(G)$, 
\begin{itemize}
\item $u$ is $\theta$-essential if $m(\theta,G\setminus u)=m(\theta,G)-1$,

\item $u$ is $\theta$-neutral if $m(\theta,G\setminus u)=m(\theta,G)$,

\item $u$ is $\theta$-positive if $m(\theta,G\setminus u)=m(\theta,G)+1$.
\end{itemize}
\end{definition}

\begin{lemma}\label{existess}\cite{Godsil1}
If $\theta$ is a root of $\mu(G,x)$, then $G$ has at least one $\theta$-essential vertex.
\end{lemma}

The well-known Gallai’s Lemma tells that if every vertex of $G$ is $\theta$-essential, then $m(\theta,G)=1$, and it was asked by Godsil \cite{Godsil1} whether it holds for non-zero $\theta$. Ku and Chen verified it, showing the following $\theta$-Gallai's theorem.

\begin{theorem}\label{essential}\cite{KC}
Let $G$ be a connected graph. If every vertex of $G$ is $\theta$-essential, then $m(\theta,G)=1$.
\end{theorem}

The definition of $\theta$-critical graph was first proposed by Godsil \cite{Godsil}, which is defined as the graph with all $\theta$-essential vertices and $m(\theta,G)=1$. By Theorem \ref{essential}, it can be simplified as follows.

\begin{definition}
Let $G$ be a graph with $\theta$ a root of $\mu(G,x)$. $G$ is said to be $\theta$-critical if each vertex of $G$ is $\theta$-essential.
\end{definition}

A further classification of vertices plays an important role in the Gallai-Edmonds
Decomposition of a graph \cite{KC}.

\begin{definition}
A vertex $u$ is said to be $\theta$-special if it is not $\theta$-essential but has a neighbor that is $\theta$-essential. 
\end{definition}

If a connected graph $G$ that is not $\theta$-critical, then it contains a $\theta$-special vertex. In fact, any $\theta$-special vertex is $\theta$-positive, see \cite[Corollary 4.3]{Godsil1}.

\begin{corollary}\label{pos}
Let $G$ be a connected graph with $\theta$ being a root of $\mu(G,x)$. If $G$ is not $\theta$-critical, then there is a $\theta$-positive vertex in $G$.
\end{corollary}

%


The following result is a partial analog to the Edmonds-Gallai structure theorem.

\begin{lemma}\label{neutral}\cite{KC}
Let $\theta$ be a root of $\mu(G,x)$ and $u$ a $\theta$-neutral vertex in $G$. Then \\
(i) if $v$ is $\theta$-essential in $G$ then it is $\theta$-essential in $G\setminus u$;\\
(ii) if $v$ is $\theta$-neutral in $G$ then it is $\theta$-neutral or $\theta$-positive in $G\setminus u$;\\
(iii) if $v$ is $\theta$-positive in $G$ then it is $\theta$--neutral or $\theta$-positive in $G\setminus u$.
\end{lemma}

\section{Proof of Theorem \ref{bound}}

In this section, we prove Theorem \ref{bound} by distinguishing whether $n$ is odd or even, based on Theorems \ref{upbound} and \ref{upbound2}. For Theorem \ref{upbound}, we provide an upper bound on $m(\theta,G)$ for all $n$-vertex graphs $G$, regardless of the parity of $n$. Note that this bound is attained only when $n$ is odd. Consequently, we proceed to consider the case when $n$ is even, which is addressed in Theorem \ref{upbound2}.

Let $n\ge 7$ be an integer. We define a set $\mathcal{F}_n$ of graphs as follows.
If $n$ is odd, then $\mathcal{F}_{n}$ is the set of graphs obtained from $K_1\cup \frac{n-1}{2}K_2$ by adding at least one edge between the vertex of $K_1$ and vertices of each $K_2$. 
If $n$ is even, then $\mathcal{F}_{n}$ is the set of graphs obtained from $\frac{n}{2}K_2$ by adding at least one edge between $w$ and vertices of each $K_2$ not containing $w$, where $w$ is a vertex in some $K_2$. 

\begin{theorem}\label{upbound}
Let $G$ be a connected graph of order $n\ge 7$.
If $\theta\ne 0$, then $m(\theta,G)\le \frac{n-3}{2}$ with equality if and only if $n$ is odd, $\theta=\pm 1$ and $G\in \mathcal{F}_{n}$.
\end{theorem}
\begin{proof}
Assume that $\theta>0$ as $m(-\theta,G)=m(\theta,G)$ by Proposition~\ref{symmetric}. 
If $\theta$ is the largest root, then $m(\theta,G)=1$ by Lemma~\ref{simple}, as desired. Assume that $\theta$ is not the largest root. Let $\rho$ denote the largest root of $\mu(G,x)$. We have by Lemma~\ref{simple} and Proposition~\ref{symmetric} that $\rho$ is simple and $-\rho$ is also a simple root of $\mu(G,x)$. So $2m(\theta,G)=m(\theta,G)+m(-\theta,G)\le n-2$.

If $m(\theta,G)=\frac{n-2}{2}$, then as $n\ge 7$, $m(\theta,G)\ge 2$ and so by Corollary \ref{pos}, there is a $\theta$-positive vertex $u$ in $G$.  It then follows from Lemma~\ref{interlacing} and Proposition~\ref{symmetric} that $m(\theta,G\setminus u)=\frac{n}{2}$ and $m(-\theta,G\setminus u)=\frac{n}{2}$, a contradiction. This proves that $m(\theta,G)\le \frac{n-3}{2}$.

Suppose that $m(\theta,G)= \frac{n-3}{2}$. Then $n$ is odd. By Corollary \ref{pos}, $G$ contains a $\theta$-positive vertex $u$, that is, $m(\theta,G\setminus u)=\frac{n-1}{2}$. By Proposition~\ref{symmetric}, $m(-\theta,G)=\frac{n-1}{2}$. This shows that $\theta$ and $-\theta$ are roots of $\mu(G\setminus u,x)$ with multiplicity $\frac{n-1}{2}$. By Lemma~\ref{simple} and Proposition~\ref{symmetric}, $G\setminus u$ contains exactly $\frac{n-1}{2}$ components, each component has $\theta$ and $-\theta$ as roots of its matching polynomial, which shows that each component is $K_2$ and $\theta=\pm 1$. As $G$ is connected, $G\in \mathcal{F}_{n}$. 

It remains to show that each graph in $\mathcal{F}_{n}$ with odd $n$ has $\pm 1$ as a root of its matching polynomial with multiplicity exactly $\frac{n-3}{2}$.
Let $H$ be a graph in $\mathcal{F}_{n}$ and $v$ the vertex of $H$ of degree at least $\frac{n-1}{2}$. Note that $H\setminus v\cong \frac{n-1}{2}K_2$ and $\mu(\frac{n-1}{2}K_2,x)$ has roots $1$ and $-1$ with multiplicity exactly $\frac{n-1}{2}$. We have by Lemma~\ref{interlacing} that $m(\pm1,H)\ge \frac{n-3}{2}$. As $n$ is odd and $H$ is connected, $m(0,H)=1$ and the largest root $\rho$ of $\mu(H,x)$ is simple with $\rho> 1$ by Lemma \ref{simple}. By Proposition~\ref{symmetric}, $-\rho$ is also a simple root of $\mu(H,x)$. Therefore, $m(\pm1,H)=\frac{n-3}{2}$. 
\end{proof}

\begin{corollary}
Let $n\ge 8$. For $\theta\ne 0$, there is no $2$-connected graph $G$ with $m(G,\theta)=\frac{n-5}{2}$.
\end{corollary}
\begin{proof}
Assume the contrary, that $G$ is a $2$-connected graph with $m(G,\theta)=\frac{n-5}{2}$. As $n\ge 8$ and $G$ is connected, $m(\theta,G)\ge 2$ and there is a $\theta$-positive vertex $u$ in $G$. Then $G\setminus u$ is connected and $m(\theta,G\setminus u)=\frac{n-3}{2}=\frac{(n-1)-2}{2}$, contradicting to Theorem~\ref{upbound}. 
\end{proof}


In the following, we consider graphs of even order.

\begin{lemma}\label{fn-}
For a graph $G\in \mathcal{F}_n$ with even $n$, $\pm 1$ is a root of $\mu(G,x)$ with multiplicity exactly $\frac{n-4}{2}$. 
\end{lemma}
\begin{proof}
Let $t=\frac{n-2}{2}$. Let $v$ denote the vertex of $G$ of degree at least $t$ and $s$ be its degree.
Note that $p_G(0)=1$, $p_G(1)=s+t$. For $i=2,\dots,t$,\[
p_G(i)={t\choose i}+{t\choose i-1}+(s-1){t-1\choose i-1}.
\]
Moreover, $p_G(t+1)=1$. Therefore, \[
\mu(G,x)=\sum_{i=0}^{t+1}(-1)^ip_G(i)x^{2t+2-2i}=(x^2-1)^{t-1}(x^4-(s+1)x^2+1)
\]
and $\pm 1$ are not roots of $x^4-(s+1)x^2+1$. This proves the lemma. 
\end{proof}

\begin{theorem}\label{upbound2}
Let $G$ be a connected graph of even order $n\ge 8$. If $\theta\ne 0$, then $m(\theta,G)\le \frac{n-4}{2}$ with equality if and only if $\theta=\pm 1$ and $G\in \mathcal{F}_{n}$. 
\end{theorem}
\begin{proof}
By Theorem~\ref{upbound} and the fact that $n$ is even, $m(\theta,G)\le \frac{n-4}{2}$. 

Suppose that $m(\theta,G)=\frac{n-4}{2}$. By Corollary~\ref{pos}, $G$ contains a $\theta$-positive vertex $u$, that is, $m(\theta,G\setminus u)=\frac{n-2}{2}$. As $n$ is even, $0$ is a root of $\mu(G\setminus u,x)$ with multiplicity one. Thus, $\pm\theta$ and $0$ are all roots of $\mu(G\setminus u,x)$. By Lemma~\ref{simple}, $G\setminus u$ contains at least $\frac{n-2}{2}$ components. As the multiplicity of root $0$ is one, there is exactly one odd component in $G\setminus u$. Hence, $G\setminus u$ contains at least $\frac{n-4}{2}$ components isomorphic to $K_2$, which implies that $\theta=\pm 1$. If the odd component contains three vertices, then it is isomorphic to either $P_3$ or $K_3$, which would imply that $\theta=\pm \sqrt{2}$ or $\theta=\pm \sqrt{3}$, a contradiction. Therefore, the odd component is exactly an isolated vertex. We conclude that $G\setminus u\cong \frac{n-2}{2}K_2\cup K_1$, and hence $G\in \mathcal{F}_{n}$.

Combining the above arguments, the result follows from Lemma \ref{fn-}.
\end{proof}

For odd $n$, let $T_n$ be the graph obtained from $K_1\cup \frac{n-1}{2}K_2$ by adding exactly one edge between the vertex of $K_1$ and a vertex of each $K_2$. For even $n$, let $T_n$ be the graph obtained from $\frac{n}{2}K_2$ by adding exactly one edge between $w$ and vertices of each $K_2$ not containing $w$, where $w$ is a vertex in some $K_2$. Obviously, $T_n$ is the only tree in $\mathcal{F}_n$. This leads to the following corollary.

\begin{corollary}
Let $T$ be a tree of order $n\ge 6$. If $\theta\ne 0$, then $m(\theta,T)\le\lfloor \frac{n-3}{2}\rfloor$ with equality if and only if $\theta=\pm 1$ and $T\cong T_n$.
\end{corollary}

\section{Graphs containing $\theta$ as a matching polynomial root}

This section begins with an investigation of $\theta$-critical graphs. We then turn to consider graphs that are not $\theta$-critical, from which we derive the maximum multiplicity of $\theta$ as a matching polynomial root. Finally, drawing inspiration from the extremal graphs, we provide an answer to Godsil's question \cite[Section 6]{Godsil1}.

We begin with a basic result, which follows immediately by Lemma~\ref{existess}.

\begin{proposition}\label{ess}
If $G$ is a connected graph with $m(\theta,G)\ge 2$, then there is a graph $G'$ with $|V(G')|<|V(G)|$ such that $m(\theta,G')=1$.
\end{proposition}

\subsection{$\theta$-critical graphs}

For any $\theta$, let \[
\mathcal{H}_\theta=\{G \mbox{ is a connected graph}:m(\theta,G)=1\}
\]
$n_\theta$ be the minimum order among all graphs in $\mathcal{H}_\theta$, and \[
\mathcal{H}_\theta'=\{G\in \mathcal{H}_\theta:|V(G)|=n_\theta\}.
\] 
By Proposition~\ref{symmetric}, $\mathcal{H}_{-\theta}=\mathcal{H}_\theta$, $n_{-\theta}=n_\theta$ and $\mathcal{H}_\theta'=\mathcal{H}_{-\theta}'$. Particularly, $\mathcal{H}_0$ is the set of all factor-critical graphs, $n_0=1$ and $\mathcal{H}_0'=\{K_1\}$. We present some examples of $n_\theta$ for non-zero $\theta$. 
If $\theta=\pm 1$, then $n_\theta=2$ and $\mathcal{H}_\theta'=\{K_2\}$. This implies that $n_\theta\ge 3$ if $\theta\ne 0,\pm 1$. 
For $\theta=\pm\sqrt{2}$ and $\theta=\pm \sqrt{3}$, we have $n_\theta=3$ with  $\mathcal{H}_\theta'=\{K_{1,2}\}$ and $\{K_3\}$, respectively. It is worth noting that different values of $\theta$ may yield the same $n_\theta$, and the $\theta$-critical graph of order $n_\theta$ is not necessarily unique.
For example, the graphs $H_1$ and $H_2$ in Fig.~\ref{fig} have the same matching polynomial $\mu(H_1,x)=\mu(H_2,x)=x^5-5x^3+4x$. This polynomial has roots $\pm \sqrt{\frac{5+\sqrt{5}}{2}}$, $\pm\sqrt{\frac{5-\sqrt{5}}{2}}$ and $0$. It can be easily checked that $H_1,H_2\in \mathcal{H}_\theta'$ with $\theta=\sqrt{\frac{5+\sqrt{5}}{2}}$. It's also worth mentioning that these two graphs form the smallest pair of connected graphs with the same matching polynomial \cite{Gutman}.

\begin{figure}[htbp]
\centering
\begin{minipage}{.495\linewidth}
\centering
\begin{tikzpicture}
\filldraw [black] (0,0) circle (2pt);
\filldraw [black] (1,0) circle (2pt);
\filldraw [black] (2,0) circle (2pt);
\draw  [black](0,0)--(-0.75,0.75)--(-1.5,0)--(0,0)--(1,0)--(2,0);
\filldraw [black] (-0.75,0.75) circle (2pt);
\filldraw [black] (-1.5,0) circle (2pt);
\node at (0,-0.5) {$H_1$};
\end{tikzpicture}
\end{minipage}
\begin{minipage}{.495\linewidth}
\centering
\begin{tikzpicture}
\filldraw [black] (0,0) circle (2pt);
\filldraw [black] (1,0) circle (2pt);
\filldraw [black] (-1,0) circle (2pt);
\draw  [black](1,0)--(0,0)--(-1,0)--(-1,1)--(0,1)--(0,0);
\filldraw [black] (-1,1) circle (2pt);
\filldraw [black] (0,1) circle (2pt);
\node at (0,-0.5) {$H_2$};
\end{tikzpicture}
\end{minipage}
\caption{The graph $H_1$ and $H_2$.}
\label{fig}
\end{figure}
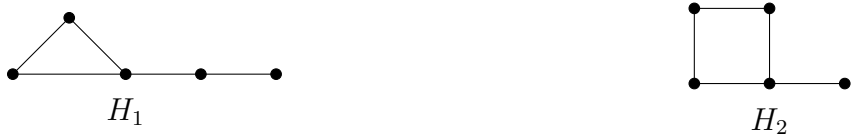

We state two straightforward results as follow. The first result is derived from Proposition~\ref{rule}~(ii), whereas the second relies on the definition of $\theta$-critical graph.

\begin{proposition}
For $G\in\mathcal{H}_\theta'$, $G-e\notin \mathcal{H}_\theta'$ for any $e\in E(G)$.
\end{proposition}

\begin{proposition}
For $G\in \mathcal{H}_\theta$, the graph obtained from $G$ by adding one vertex adjacent to some vertices of $G$ can not be $\theta$-critical.
\end{proposition}

It is therefore natural to consider the existence of $\theta$ as a matching polynomial root. Equivalently, one may ask whether $\mathcal{H}_\theta$ and $\mathcal{H}_\theta'$ are empty.

A totally real algebraic integer is a root of some real-rooted monic polynomial with integer coefficients.
It's shown in \cite{Salez} that every totally real algebraic integer is a tree eigenvalue. As the matching polynomial of a tree is exactly its characteristic polynomial, every totally real algebraic integer is a root of the matching polynomial of a tree. This shows that, if $\theta$ is a totally real algebraic integer, then we have by Proposition \ref{ess} and the definition of $n_\theta$, there is a graph of order $n_\theta$ that is $\theta$-essential. This is precisely the statement of the next proposition.

\begin{proposition}
For any totally real algebraic integer $\theta$, both $\mathcal{H}_\theta$ and $\mathcal{H}_\theta'$ are not empty.
\end{proposition}

\subsection{Graphs that are not $\theta$-critical}

We now turn to consider graphs that are not $\theta$-critical and focus on the order of such graphs.

\begin{lemma}\label{order}
Let $G$ be a connected graph with $m(\theta,G)\ge k\ge 2$. Then $|V(G)|\ge (k+1)n_\theta+1$. 
\end{lemma}
\begin{proof}
Assume that $G$ has the minimum order among all graphs satisfying the condition. 
We prove by induction on $k$. 

Suppose that $k=2$. By Corollary~\ref{pos}, $G$ contains a $\theta$-positive vertex $u$. Then $m(\theta,G\setminus u)\ge 3$ and $G\setminus u$ is disconnected, as otherwise $G\setminus u$ is a connected graph of less order such that $m(\theta,G\setminus u)\ge k$, contradicting the assumption of $G$. Similarly, each component $G'$ of $G\setminus u$ satisfies that $m(\theta,G')\le 1$. Let $G_1,\dots,G_s$ denote the components of $G\setminus u$ with $\theta$ being a root of their matching polynomial. Then $s\ge 3$ and $|V(G_i)|\ge n_\theta$ for $i=1,\dots,s$. So \[
|V(G)|\ge \sum_{i=1}^s|V(G_i)|+1\ge 3n_\theta+1.
\]
This shows the result for $k=2$.

Assume in the following that $k\ge 3$. By Corollary~\ref{pos} again, $G$ contains a $\theta$-positive vertex $u$. Then $m(\theta,G\setminus v)\ge k+1$ and $G\setminus v$ is disconnected. Let $G_1,\dots,G_s$ denote the components of $G\setminus v$ with $\theta$ as a matching polynomial root and $t_i=m(\theta,G_i)$ for $i=1,\dots,s$. Then $t_i\le k-1$ for $i=1,\dots,s$ and $\sum_{i=1}^st_i\ge k+1$. We have by induction hypothesis, \[
|V(G_i)|\ge \begin{cases}
n_\theta,&\mbox{ if }t_i=1,\\
(t_i+1)n_\theta+1,&\mbox{ if }t_i\ge 2.
\end{cases}
\]
Let $s'=|\{t_i:t_i\ge 2\}|$ and assume that $t_i\ge 2$ for $1\le i\le s'$. Then $\sum_{i=1}^{s'}t_i\ge k+1-s+s'$ and so 
\begin{align*}
|V(G)|\ge \sum_{i=1}^s|V(G_i)|+1&\ge \sum_{i=1}^{s'}|V(G_i)|+(s-s')n_\theta+1\\
&\ge \sum_{i=1}^{s'}((t_i+1)n_\theta+1)+(s-s')n_\theta+1\ge (k+1)n_\theta+1,
\end{align*}
proving the result. 
\end{proof}

\begin{theorem}\label{orderr}
Let $G$ be a connected graph with $m(\theta,G)\ge k\ge 1$. If $G$ is not $\theta$-critical, then $|V(G)|\ge (k+1)n_\theta+1$. 
\end{theorem}
\begin{proof}
By Lemma \ref{order}, it remains to show the case for $k=1$. As $G$ is not $\theta$-critical, there is a $\theta$-positive vertex $u$ by Corollary \ref{pos}. Then $m(\theta,G\setminus u)\ge 2$. If $G\setminus u$ is connected, or if it is  disconnected but contains a component $G'$ with $m(\theta,G')\ge 2$, then the result follows from Lemma \ref{order}. If $G\setminus u$ is disconnected and each component $H$ satisfies $m(\theta,H)\le 1$, then there are at least two components containing $\theta$ as a matching polynomial root. Therefore, $|V(G\setminus u)|\ge 2n_{\theta}$, and thus $|V(G)|\ge 2n_\theta+1$. 
\end{proof}

The following corollary is immediate.

\begin{corollary}\label{critical}
Each graph in $\mathcal{H}_\theta$ of order $n_\theta$ is $\theta$-critical. Moreoever, if $G$ is a graph  of order $n$ with a matching polynomial root $\theta$, where $n_\theta\le n\le 2n_\theta$, then $G$ is $\theta$-critical.
\end{corollary}

Actually, it's possible that there is no $\theta$-critical graphs of order $n$, where $n_\theta<n\le 2n_\theta$. For example, for $\theta=1$, $n_\theta=2$ and there is no $1$-critical graphs of order $3$ or $4$ (see in Section 5).
On the other hand, for $\theta=\sqrt{3}$, $n_\theta=3$, $K_{1,3}$ and $P_5$ are $\sqrt{3}$-critical graphs of order $4$ and $5$, respectively.
However, the converse of the second part of Corollary \ref{critical} is not true, see Theorem \ref{1}. 

\subsection{Proof of Theorem \ref{essbound}}

For integer $n$ with $n\equiv 1\pmod {n_\theta}$, let $\mathcal{H}_\theta^n$ be the set of graphs obtained from $K_1\cup G_1\cup \dots \cup G_t$ by adding at least one edge between the vertex of $K_1$ and each $G_i$, where $G_1,\dots,G_t\in \mathcal{H}_\theta'$ and $t=\frac{n-1}{n_\theta}$. 

Now, we're ready to show the maximum multiplicity of $\theta$ as a matching polynomial root for general $\theta$.

\begin{proof}[Proof of Theorem \ref{essbound}]
The upper bound follows by Lemma~\ref{order}. It suffices to prove that $m(\theta,G)=\frac{n-n_\theta-1}{n_\theta}$ with equality if and only if $n\equiv 1\pmod {n_\theta}$ and $G\in \mathcal{H}_{\theta}^n$.

Suppose that $m(\theta,G)=\frac{n-n_\theta-1}{n_\theta}:=k$. Then $n\equiv 1\pmod {n_\theta}$. As $G$ is not $\theta$-critical, there is a $\theta$-positive vertex $u$ in $G$ by Corollary~\ref{pos}. Then $m(\theta,G\setminus u)=k+1$. 
Let $G_1,\dots,G_s$ denote the components of $G\setminus u$. It's possible for $s=1$.
Let $t_i=m(\theta,G_i)$ for $i=1,\dots,s$ and $s'=|\{t_i:t_i\ge 2\}|$. Assume that $t_i\ge 2$ for $1\le i\le s'$. Then $\sum_{i=1}^{s'}t_i=k+1-s+s'$. If $s'\ge 1$, then by Lemma \ref{order},
\[
n-1=\sum_{i=1}^t|V(G_i)|\ge\sum_{i=1}^{s'}((t_i+1)n_\theta+1)+(s-s')n_\theta=(k+1+s')n_\theta+s'=n-1+s'(n_\theta+1)>n-1,
\]
a contradiction. Hence, $t_i=1$ for $i=1,\dots,s$ and thus $s=k+1$. Therefore, each $G_i$ has order $n_\theta$, $|V(G_i)|=n_\theta$, i.e., $G_i\in \mathcal{H}_\theta'$ for all $i$. As $G$ is connected, we conclude that $G\in \mathcal{H}_{\theta}^n$.

It remains to show that each graph $H\in \mathcal{H}_\theta^n$ satisfies that $m(\theta,H)=\frac{n-n_\theta-1}{n_\theta}$. Let $v$ denote the vertex of $K_1$ in $H$. Note that $H\setminus v$ consists of $\frac{n-1}{n_\theta}$ components, all of which belong to $\mathcal{H}_\theta$. So $m(\theta,H\setminus v)=\frac{n-1}{n_\theta}$. It then follows by Lemma~\ref{interlacing} that $m(\theta,H)\ge \frac{n-n_\theta-1}{n_\theta}$. By Lemma \ref{order}, $m(\theta,H)\le \frac{n-n_\theta-1}{n_\theta}$ and so $m(\theta,H)= \frac{n-n_\theta-1}{n_\theta}$.
\end{proof}

Actually, Theorem \ref{upbound} can be viewed as a special case of Theorem \ref{essbound} for $\theta=\pm 1$. For any nonzero $\theta\ne \pm 1$, we have $n_\theta\ge 3$, and thus the corollary follows from the fact that $\mathcal{H}_{\pm \sqrt{2}}'=\{P_3\}$ and $\mathcal{H}_{\pm\sqrt{3}}'=\{K_3\}$.

\begin{corollary}
Let $G$ be a connected graph that is not $\theta$-critical. If $\theta\ne 0,\pm 1$, then $m(\theta,G)\le \frac{n-4}{3}$ with equality if and only if $\theta=\pm \sqrt{2}$ or $\pm\sqrt{3}$, $n\equiv 1\pmod 3$ and $G\in \mathcal{H}_{\theta}^n$. 
\end{corollary}



%

Noting that all extremal graphs (i.e., graphs attaining the equality) contain cut-vertices, we are led to consider the extension of the result to $t$-connected graphs.

Let $t\ge 1$ and $n\ge (t+2)n_\theta+t$ with $n\equiv t\pmod{n_\theta}$, and set $s=\frac{n-t}{n_\theta}$. Let $H=tK_1\cup G_1\cup \dots \cup G_s$ with $G_1,\dots,G_s\in \mathcal{H}_\theta'$ and $A=V(tK_1)$ be the set of $t$ isolated vertices. Let $\mathcal{H}_{\theta}^{n,t}$ denote the set of graphs obtained from $H$ by: \\
(i) for each $v\in A$ and $i=1,\dots,s$, adding at least one edge between $v$ and  $V(G_i)$, and \\
(ii) adding an arbitrary set of edges within $A$.\\
In particularly, when $t=1$ we return to $\mathcal{H}_\theta^{n,1}=\mathcal{H}_\theta^n$.
%
%
%

\begin{theorem}
Let $G$ be a $t$-connected graph that is not $\theta$-critical. Then $m(\theta,G)\le \frac{n-(n_\theta+1)t}{n_\theta}$ with equality if and only if $n\equiv t\pmod {n_\theta}$ and $G\in \mathcal{H}_\theta^{n,t}$.
\end{theorem}
\begin{proof}
We prove by induction on $t$. If $t=1$, then the result follows by Theorem~\ref{essbound}. Assume that $t\ge 2$ and the result is true for all $(t-1)$-connected graphs. 

As $G$ is not $\theta$-critical, we have by Corollary~\ref{pos} that there is a $\theta$-positive vertex $u$ such that $m(\theta,G\setminus u)=m(\theta,G)+1$. Note that $G\setminus u$ is $(t-1)$-connected. Then by the induction hypothesis, $m(\theta,G\setminus u)\le \frac{n-1-(n_\theta+1)(t-1)}{n_\theta}$ and so 
\begin{equation}\label{eq1}
m(\theta,G)=m(\theta,G\setminus u)-1\le  \frac{n-1-(n_\theta+1)(t-1)}{n_\theta}-1=\frac{n-(n_\theta+1)t}{n_\theta}.
\end{equation}

Suppose that $m(\theta,G)=\frac{n-(n_\theta+1)t}{n_\theta}$. Then $n\equiv t\pmod {n_\theta}$ and $m(\theta,G)\ge 2$. By Corollary~\ref{pos}, $G$ contains a $\theta$-positive vertex $u$. Then $m(\theta,G\setminus u)=\frac{n-(n_\theta+1)t}{n_\theta}+1=\frac{n-1-(n_\theta+1)(t-1)}{n_\theta}$. Note that $G\setminus u$ is a $(t-1)$-connected graph of order $n-1$. By the induction hypothesis, $G\setminus u\in \mathcal{H}_\theta^{n-1,t-1}$. As the connectivity of $G\setminus u$ is equal to $t-1$ and $G$ is $t$-connected, we have $G\in \mathcal{H}_\theta^{n,t}$. 

It remains to show that each graph $H\in \mathcal{H}_\theta^{n,t}$ satisfies $m(\theta,H)=\frac{n-(n_\theta+1)t}{n_\theta}$. Let $W$ denote the vertex cut of $H$ on $t$ vertices such that $H\setminus W$ consists of components in $\mathcal{H}_\theta$. Then $m(\theta,H\setminus W)=\frac{n-t}{n_\theta}$ and so, by Lemma \ref{interlacing}, $m(\theta,H)\ge \frac{n-t}{n_\theta}-t=\frac{n-(n_\theta+1)t}{n_\theta}$. On the other hand, $m(\theta,H)\le \frac{n-(n_\theta+1)t}{n_\theta}$ by Eq.~\eqref{eq1}. Therefore, $m(\theta,H)=\frac{n-(n_\theta+1)t}{n_\theta}$. This proves the theorem.
\end{proof}

The extension of Theorem \ref{upbound} to $t$-connected graphs is shown as follows.

\begin{corollary}
Let $G$ be a $t$-connected graph of order $n\ge 3t+2$. If $\theta\ne 0$, then $m(\theta,G)\le \frac{n-3t}{2}$ if and only if $n\equiv t\pmod 2$, $\theta=\pm 1$ and $G\in \mathcal{H}_1^{n,t}$.
\end{corollary}

\subsection{Proof of Theorem \ref{spec}}

We prove Theorem \ref{spec} in this subsection by construction. 
In the subsequent discussion, we aim to establish Proposition \ref{qk}, which implies Theorem \ref{spec}.

Two vertices $u$ and $v$ of $G$ are called twin vertices if they are adjacent and share exactly the same neighbors in $G$, excluding each other.

Let $G_0$ be a graph obtained from a graph $G_0'$ in $\mathcal{H}_\theta'$ by adding a twin vertex $v'$ of some vertex $v$ in $G_0'$.
Let $k$ be an integer with $k\ge 2$, $u$ be a vertex and $G_1,\dots, G_k\in \mathcal{H}_\theta'$. Let
$\mathcal{Q}_\theta^k$ denote the set of graphs obtained from $\{u\}\cup G_0\cup G_1\cup\dots \cup G_k$ by:\\
(i) adding edges $uv$ and $uv'$, and \\
(ii) for each $i=1,\dots,k$, adding at least one edge between $u$ and $V(G_i)$.\\ Particularly, when $\theta=0$, we have $\mathcal{Q}_0^k=\left\{K_1\vee (K_2\cup kK_1)\right\}$. See Fig. \ref{example} for an example of a graph in $\mathcal{Q}_\theta^k$ with $\theta=\sqrt{3}$ and $k=2$.

\begin{proposition}\label{qk}
For $G\in \mathcal{Q}_\theta^k$, $v$ is $\theta$-positive but not $\theta$-special in $G$.
\end{proposition}
\begin{proof}
We firstly claim that $\theta$ is not a root of $\mu(G_0,x)$. Otherwise, we have by Lemma \ref{order} that $m(\theta,G_0)=1$. As $G_0\setminus v'=G_0'\in \mathcal{H}_\theta'$, $v'$ is $\theta$-neutral. Similarly, $v$ is $\theta$-neutral. It then follows from Lemma \ref{neutral} that $v$ is $\theta$-neutral or $\theta$-positive in $G_0'$. However, $G_0'$ is $\theta$-critical, any vertex in $G_0'$ is $\theta$-essential, a contradiction. 

Next, we show that $m(\theta,G)=k-1$. If $\theta=0$, then $G=K_1\vee (K_2\cup kK_1)$, which shows that there are exactly $k-1$ vertices missing the maximum matching, and thus $m(\theta,G)=k-1$. Now we assume that $\theta\ne 0$. In this case, $n_\theta\ge 2$.
As $|V(G)|=(k+1)n_\theta+2$, Lemma \ref{order} implies that $m(\theta,G)\le k$. 
As $G\setminus u=G_0\cup G_1\cup\dots \cup G_k$ and $G_1,\dots,G_k\in \mathcal{H}_\theta'$, it follows that $m(\theta,G\setminus u)=k$. 
Suppose, for contradiction, that $m(\theta,G)=k$. Then $u$ is $\theta$-neutral in $G$. Let $w$ be a vertex in $G_1$. As $G_1\in \mathcal{H}_\theta'$, we have $m(\theta,G_1\setminus w)=0$, and thus $m(\theta,G\setminus uw)=k-1$. This implies that $w$ is $\theta$-essential in $G\setminus u$. By Lemma \ref{neutral}, $w$ is $\theta$-essential in $G$, which shows that $u$ is $\theta$-special in $G$. Hence, $u$ is $\theta$-positive in $G$, contradicting to the assumption that $u$ is $\theta$-neutral. Therefore, $m(\theta,G)=k-1$, and $u$ is $\theta$-positive in $G$, as claimed.

Note that $G\setminus v\in \mathcal{H}_\theta^{(k+1)n_\theta+1}$. We have $m(\theta,G\setminus v)=k$ and so $v$ is $\theta$-positive in $G$. Similarly, $v'$ is $\theta$-positive in $G$. For any $z\in V(G_0)\setminus \{v,v'\}$, as $m(\theta,G\setminus uz)\ge k$, $m(\theta,G\setminus z)\ge k-1$ and so $z$ is $\theta$-neutral or $\theta$-positvie. So $v$ is not $\theta$-special.
\end{proof}

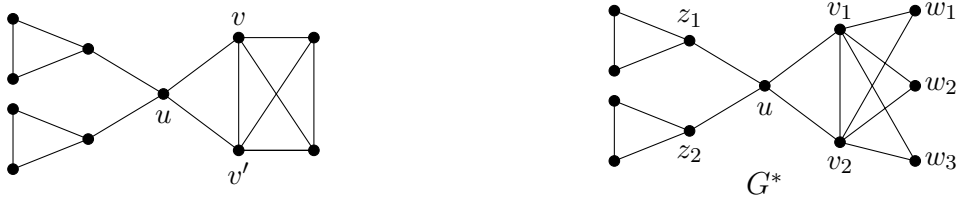
\begin{figure}
\centering
\begin{minipage}{.495\linewidth}
\centering
\begin{tikzpicture}
\filldraw [black] (0,0) circle (2pt);
\filldraw [black] (1,0.75) circle (2pt);
\filldraw [black] (1,-0.75) circle (2pt);
\filldraw [black] (2,-0.75) circle (2pt);
\filldraw [black] (2,0.75) circle (2pt);
\filldraw [black] (-1,0.6) circle (2pt);
\filldraw [black] (-1,-0.6) circle (2pt);
\filldraw [black] (-2,-1) circle (2pt);
\filldraw [black] (-2,-0.2) circle (2pt);
\filldraw [black] (-2,1) circle (2pt);
\filldraw [black] (-2,0.2) circle (2pt);
\draw  [black](1,-0.75)--(0,0)--(1,0.75)--(2,-0.75)--(1,-0.75)--(2,0.75)--(2,-0.75)--(1,-0.75);
\draw [black] (1,-0.75)--(1,0.75)--(2,0.75);
\draw[black] (0,0)--(-1,-0.6)--(-2,-0.2)--(-2,-1)--(-1,-0.6);
\draw[black] (0,0)--(-1,0.6)--(-2,0.2)--(-2,1)--(-1,0.6);
\node at (0,-0.3) {\small $u$};
\node at (1,1) {\small $v$};
\node at (1,-1.05) {\small $v'$};
\end{tikzpicture}
\end{minipage}
\begin{minipage}{.495\linewidth}
\centering
\begin{tikzpicture}
\filldraw [black] (0,0) circle (2pt);
\filldraw [black] (1,0.75) circle (2pt);
\filldraw [black] (2,0) circle (2pt);
\filldraw [black] (1,-0.75) circle (2pt);
\filldraw [black] (2,-1) circle (2pt);
\filldraw [black] (2,1) circle (2pt);
\filldraw [black] (-1,0.6) circle (2pt);
\filldraw [black] (-1,-0.6) circle (2pt);
\filldraw [black] (-2,-1) circle (2pt);
\filldraw [black] (-2,-0.2) circle (2pt);
\filldraw [black] (-2,1) circle (2pt);
\filldraw [black] (-2,0.2) circle (2pt);
\draw  [black](1,-0.75)--(0,0)--(1,0.75)--(2,1)--(1,-0.75)--(1,0.75)--(2,0)--(1,-0.75)--(2,-1)--(1,0.75);
\draw[black] (0,0)--(-1,-0.6)--(-2,-0.2)--(-2,-1)--(-1,-0.6);
\draw[black] (0,0)--(-1,0.6)--(-2,0.2)--(-2,1)--(-1,0.6);
\node at (0,-0.3) {\small $u$};
\node at (1,1) {\small $v_1$};
\node at (1,-1.05) {\small $v_2$};
\node at (2.35,1) {\small $w_1$};
\node at (2.35,0) {\small $w_2$};
\node at (2.35,-1) {\small $w_3$};
\node at (-1,0.9) {\small $z_1$};
\node at (-1,-0.9) {\small $z_2$};
\node at (0,-1.3) {$G^*$};
\end{tikzpicture}
\end{minipage}
\caption{A graph in $\mathcal{Q}_{\sqrt{3}}^2$ and the graph $G^*$ (left to right).}
\label{example}
\end{figure}

Graphs in $\mathcal{Q}_\theta^k$ are not the only graphs satisfying Theorem \ref{spec}. For instance, consider the graph $G^*$ shown in Fig. \ref{example}.
By a direct calculation, \[
\mu(G^*,x)=x^{12}-17x^{10}+97x^8-227x^6+198x^4-36x^2,
\]
\[
\mu(G^*\setminus u,x)=x^{11}-13x^9+57x^7-99x^5+54x^3,
\]
\[
\mu(G^*\setminus v_1,x)=x^{11}-12x^9+47x^7-66x^5+18x^3
\]
\[
\mu(G^*\setminus w_1,x)=x^{11}-15x^9+75x^7-149x^5+100x^3-12x
\]
and
\[
\mu(G^*\setminus z_1,x)=x^{11}-14x^9+60x^7-96x^5+55x^3-6x.
\]
Then $m(\sqrt{3},G^*)=1$, $m(\sqrt{3},G^*\setminus u)=2$, $m(\sqrt{3},G^*\setminus v_1)=2$, $m(\sqrt{3},G^*\setminus w_1)=1$ and $m(\sqrt{3},G^*\setminus z_1)=0$. So $u$ and $v_1$ are $\sqrt{3}$-positive, $w_1$ is $\sqrt{3}$-neutral and $z_1$ is $\sqrt{3}$-essential. Note that $G^*\setminus v_2\cong G^*\setminus v_1$, $G^*\setminus w_i\cong G^*\setminus w_1$ for $i=2,3$ and $G^*\setminus z_2\cong G^*\setminus z_1$. So $u$ is $\sqrt{3}$-special and $v_1$ is not $\sqrt{3}$-special. This shows that a $\theta$-positive vertex does not need to be $\theta$-special.

\section{Further results on $1$-critical graphs}

In this section, we prove Theorem \ref{1} and present a construction that yields a new $1$-critical graph from a given one.

\subsection{Existence of $1$-critical trees}

We prove the first part of Theorem \ref{1} in this subsection. The proof is constructive, relying on the construction provided in Propositions \ref{wn}, \ref{fn} and \ref{fn*}.


A pendant vertex of a graph $G$ is a vertex of degree one in $G$, and an edge incident to a pendant vertex is called a pendant edge. A quasi-pendant vertex is the neighbor of a pendant vertex.

For $n\ge 6$, let $W_n$ denote the graph obtained from the path $P_{n-2}$ by attaching a pendant vertex to each quasi-pendant vertex of $P_{n-2}$, see Fig.~\ref{WY}.
For $n\ge 4$, let $Y_n$ denote the graph obtained from the path $P_{n-1}$ by attaching a pendant vertex to a quasi-pendant vertex of $P_{n-1}$, also see Fig.~\ref{WY}.
\begin{figure}[htbp]
\centering
\begin{minipage}{.495\linewidth}
\centering
\begin{tikzpicture}
\filldraw [black] (-3.5,0) circle (2pt);
\filldraw [black] (-2.5,0) circle (2pt);
\filldraw [black] (-1.5,0) circle (2pt);
\filldraw [black] (0,0) circle (2pt);
\filldraw [black] (1,0) circle (2pt);
\filldraw [black] (2,0) circle (2pt);
\draw  [black](0,0)--(1,0)--(2,0);
\draw [black] (1,0)--(1.75,0.75);
\draw [dashed] (0,0)--(-1.5,0);
\draw [black] (-3.5,0)--(-2.5,0)--(-1.5,0);
\draw [black] (-3.25,0.75)--(-2.5,0);
\filldraw [black] (1.75,0.75) circle (2pt);
\filldraw [black] (-3.25,0.75) circle (2pt);
\node at (-0.75,-0.5) {$W_n$};
\end{tikzpicture}
\end{minipage}
\begin{minipage}{.495\linewidth}
\centering
\begin{tikzpicture}
\filldraw [black] (-3.5,0) circle (2pt);
\filldraw [black] (-2.5,0) circle (2pt);
\filldraw [black] (-1.5,0) circle (2pt);
\filldraw [black] (0,0) circle (2pt);
\filldraw [black] (1,0) circle (2pt);
\filldraw [black] (2,0) circle (2pt);
\draw  [black](0,0)--(1,0)--(2,0);
\draw [dashed] (0,0)--(-1.5,0);
\draw [black] (-3.5,0)--(-2.5,0)--(-1.5,0);
\draw [black] (-3.25,0.75)--(-2.5,0);
\filldraw [black] (-3.25,0.75) circle (2pt);
\node at (-0.75,-0.5) {$Y_n$};
\end{tikzpicture}
\end{minipage}
\caption{The graph $W_n$ and $Y_n$ (left to right).}
\label{WY}
\end{figure}
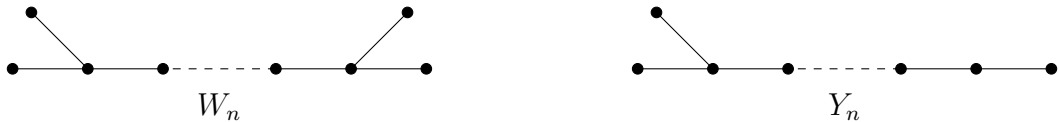

\begin{proposition}\label{wn}
For $n\ge 6$ with $n\equiv 0\pmod 3$, $W_n$ is $1$-critical.
\end{proposition}
\begin{proof}
If $n=6$, then it can be obtained by direct calculation that $\mu(W_n,x)=x^6-5x^4+ 4x^2$, so $1$ is a root of $\mu(W_n,x)$ with multiplicity one. Moreoever, it can be verified that each vertex of $W_6$ is $1$-essential. Therefore, $W_6$ is $1$-critical. Assume in the following that $n\ge 7$.

Let $e=uv$ be a pendant edge in $W_n$ with $u$ being the quasi-pendant vertex. It's obvious that $G-e\cong K_1\cup Y_{n-1}$ and $G\setminus uv\cong K_1\cup Y_{n-3}$. 
By Theorem~\ref{rule}, we have 
\begin{equation}\label{eqwn}
\mu(W_n,x)=x\mu(Y_{n-1},x)-x\mu(Y_{n-3},x).
\end{equation}
Let $u'$ be a neighbor of $u$ that is not pendant in $W_n$ and $f=uu'$. Note that $G-f\cong Y_{n-3}\cup P_3$ and $G\setminus uu'\cong 2K_1\cup Y_{n-4}$. As $\mu(P_3,x)=x(x^2-2)$, we have by Theorem~\ref{rule} again that \[
\mu(W_n,x)=x(x^2-2)\mu(Y_{n-3},x)-x^2\mu(Y_{n-4},x).
\]
From Eq.~\eqref{eqwn}, $x\mu(Y_{n-3},x)=x\mu(Y_{n-1},x)-\mu(W_n,x)$ and so \[
\mu(W_n,x)=(x^2-2)(x\mu(Y_{n-1},x)-\mu(W_n,x))-x^2\mu(Y_{n-4},x).
\]
This shows that 
\[
\mu(W_n,1)=-(\mu(Y_{n-1},1)-\mu(W_n,1))-\mu(Y_{n-4},1)=\mu(W_n,1)-\mu(Y_{n-1},1)-\mu(Y_{n-4},1),
\]
i.e., 
\begin{equation}\label{eqyn}
\mu(Y_n,1)=-\mu(Y_{n-3},1).
\end{equation}

By an easy calculation, we have \[
\mu(Y_5,x)=x^5 - 4x^3 + 2x, 
\]
\[
\mu(Y_6,x)=x^6 - 5x^4 + 5x^2
\]
and \[
\mu(Y_7,x)=x^7 - 6x^5 + 9x^3 - 2x.
\]
It then follows that $\mu(Y_5,1)=-1$, $\mu(Y_6,1)=1$ and $\mu(Y_7,1)=2$. So by Eq.~\eqref{eqyn}, 
\begin{equation}\label{eqynn}
\mu(Y_n,1)=\begin{cases}
(-1)^{\frac{n-1}{3}}\cdot 2&\mbox{ if }n\equiv 1\pmod 3,\\
(-1)^{\lfloor\frac{n}{3}\rfloor}&\mbox{ otherwise.}
\end{cases}
\end{equation}
This shows that $\mu(Y_n,1)\ne 0$ for any $n\ge 5$, i.e., $1$ is not a root of $\mu(Y_n,x)$.
From Eq.~\eqref{eqwn} again, we have 
\begin{equation}\label{eqwnn}
\mu(W_n,1)=\mu(Y_{n-1},1)-\mu(Y_{n-3},1)=\begin{cases}
0&\mbox{ if }n\equiv 0\pmod 3,\\
(-1)^{\frac{n+2}{3}}&\mbox{ if }n\equiv 1\pmod 3,\\
(-1)^{\frac{n-2}{3}}\cdot 3&\mbox{ if }n\equiv 2\pmod 3.
\end{cases}
\end{equation}
Note that for any $w\in V(W_n)$, $W_n\setminus w$ consists of either $P_t$ with $t=1,3$ or $Y_s$ with $4\le s\le n-1$. Recall that $1$ is not a root of $\mu(Y_s,x)$. So any vertex of $W_n$ is $1$-essential if $n\equiv 0\pmod 3$. By Theorem \ref{essential}, $m(1,W_n)=1$ if $n\equiv 0\pmod 3$. Therefore, $W_n$ is $1$-critical if $n\equiv 0\pmod 3$.
\end{proof}

For convenience, we assume that $P_3=Y_3$.
By an easy calculation, $\mu(Y_3,x)=x^3-2x$ and $\mu(Y_4,x)=x^4-3x$, satisfying Eq.~\eqref{eqynn}. So we have the following result.

\begin{lemma}\label{y}
For $n\ge 3$, $\mu(Y_n,1)=\begin{cases}
(-1)^{\frac{n-1}{3}}\cdot 2&\mbox{ if }n\equiv 1\pmod 3,\\
(-1)^{\lfloor\frac{n}{3}\rfloor}&\mbox{ otherwise.}
\end{cases}$ 
\end{lemma}

For $n\ge 7$, let $R_n$ denote the graph obtained from the path $P_{n-3}:=v_0v_1\dots v_{n-4}$ by attaching a pendant vertex to each vertex in $\{v_1,v_2,v_3\}$, see Fig.~\ref{figrn}.
For $n\ge 8$, let $R_n^*$ denote the graph obtained from $P_{n-4}:=v_0v_1\dots v_{n-5}$ by attaching two pendant vertices to each vertex in $\{v_1,v_2\}$, also see Fig. \ref{figrn}.

\begin{figure}[htbp]
\centering
\begin{minipage}{.495\linewidth}
\centering
\begin{tikzpicture}
\filldraw [black] (-5.5,0) circle (2pt);
\filldraw [black] (-4.5,0) circle (2pt);
\filldraw [black] (-3.5,0) circle (2pt);
\filldraw [black] (-2.5,0) circle (2pt);
\filldraw [black] (-1.5,0) circle (2pt);
\filldraw [black] (0,0) circle (2pt);
\filldraw [black] (1,0) circle (2pt);
\draw  [black](0,0)--(1,0);
\draw [dashed] (0,0)--(-1.5,0);
\draw [black] (-5.5,0)--(-4.5,0)--(-3.5,0)--(-2.5,0)--(-1.5,0);
\draw [black] (-5.25,0.75)--(-4.5,0);
\filldraw [black] (-3.5,1) circle (2pt);
\filldraw [black] (-5.25,0.75) circle (2pt);
\filldraw [black] (-2.5,1) circle (2pt);
\draw [black] (-3.5,0)--(-3.5,1);
\draw [black] (-2.5,0)--(-2.5,1);
\node at (-4.5,-0.35) {\small $v_1$};
\node at (-3.5,-0.35) {\small $v_2$};
\node at (-2.5,-0.35) {\small $v_3$};
\node at (-1.2,-0.7) {$R_n$};
\end{tikzpicture}
\end{minipage}
\begin{minipage}{.495\linewidth}
\centering
\begin{tikzpicture}
\filldraw [black] (-4.5,0) circle (2pt);
\filldraw [black] (-3.5,0) circle (2pt);
\filldraw [black] (-2.5,0) circle (2pt);
\filldraw [black] (-1.5,0) circle (2pt);
\filldraw [black] (0,0) circle (2pt);
\filldraw [black] (1,0) circle (2pt);
\draw  [black](0,0)--(1,0);
\draw [dashed] (0,0)--(-1.5,0);
\draw [black] (-4.5,0)--(-3.5,0)--(-2.5,0)--(-1.5,0);
\draw [black] (-4.25,0.75)--(-3.5,0)--(-4.25,-0.75);
\filldraw [black] (-3.25,0.75) circle (2pt);
\filldraw [black] (-1.75,0.75) circle (2pt);
\filldraw [black] (-4.25,0.75) circle (2pt);
\filldraw [black] (-4.25,-0.75) circle (2pt);
\draw [black] (-3.25,0.75)--(-2.5,0)--(-1.75,0.75);
\node at (-3.5,-0.35) {\small $v_1$};
\node at (-2.5,-0.35) {\small $v_2$};
\node at (-1,-0.7) {$R_n^*$};
\end{tikzpicture}
\end{minipage}
\caption{The graphs $R_n$ and $R_n^*$ (left to right).}
\label{figrn}
\end{figure}
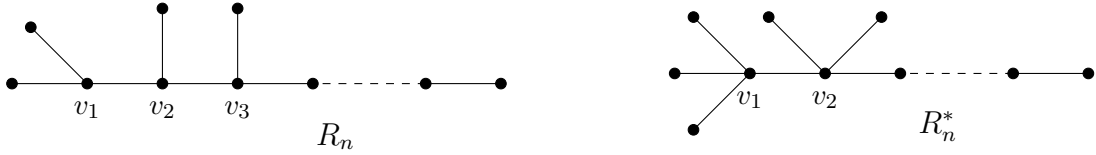

\begin{lemma}\label{rn}
(i) For $n\ge 7$, $\mu(R_n,1)=\begin{cases}
(-1)^{\frac{n-2}{3}}\cdot 2&\mbox{ if }n\equiv 2\pmod 3,\\
(-1)^{\lceil\frac{n}{3}\rceil+1}&\mbox{ otherwise.}
\end{cases}$ \\
(ii) For $n\ge 8$, $\mu(R_n^*,1)=\begin{cases}
(-1)^{\frac{n-3}{3}}\cdot 2&\mbox{ if }n\equiv 0\pmod 3,\\
(-1)^{\frac{n-1}{3}}&\mbox{ if }n\equiv 1\pmod 3,\\
(-1)^{\frac{n-2}{3}}\cdot 3&\mbox{ if }n\equiv 2\pmod 3.
\end{cases}$
\end{lemma}
\begin{proof}
We only prove (i), and the proof of (ii) is similar and thus omitted.
Using the notation introduced above, we prove the result by induction. 
By calculation, \[
\mu(R_7,x)=x^7-6x^5+8x^3-2x
\]
and 
\[
\mu(R_8,x)=x^8 - 7x^6 + 12x^4 - 4x^2.
\]
Then $\mu(R_7,1)=1$ and $\mu(R_8,1)=2$.  Assume in the following that $n\ge 9$ and statement holds for any $R_t$ with $7\le t<n$.

As $R_n-v_{n-5}v_{n-4}\cong K_1\cup R_{n-1}$ and $R_n\setminus v_{n-5}v_{n-4}\cong R_{n-2}$, we have by Proposition~\ref{rule} that \[
\mu(R_n,x)=x\mu(R_{n-1},x)-\mu(R_{n-2},x)
\]
and so
\[
\mu(R_n,1)=\mu(R_{n-1},1)-\mu(R_{n-2},1)=\begin{cases}
(-1)^{\frac{n-3}{3}}\cdot 2-(-1)^{\frac{n}{3}+1}=(-1)^{\frac{n}{3}+1}&\mbox{ if }n\equiv 0\pmod 3,\\
(-1)^{\frac{n-1}{3}+1}-(-1)^{\frac{n-4}{3}}\cdot 2=(-1)^{\frac{n+2}{3}+1}&\mbox{ if }n\equiv 1\pmod 3,\\
(-1)^{\frac{n+1}{3}+1}-(-1)^{\frac{n-2}{3}+1}=(-1)^{\frac{n-2}{3}}\cdot 2&\mbox{ if }n\equiv 2\pmod 3,\\
\end{cases}
\]
as desired.
%
%
\end{proof}

For $n\ge 10$, let $F_{n}$ denote the graph obtained from the path $P_{n-4}:=v_0v_1\dots v_{n-5}$ by attaching a pendant vertex to each vertex in $\{v_1,v_2,v_{3},v_{n-6}\}$, see Fig.~\ref{figfn}.
For $n\ge 11$, let $F_n^*$ denote the graph obtained from the path $P_{n-6}:=v_0v_1\dots v_{n-7}$ by attaching two pendant vertices to each vertex in $\{v_1,v_2,v_{n-8}\}$, also see Fig. \ref{figfn}.

\begin{figure}[htbp]
\centering
\begin{minipage}{.495\linewidth}
\centering
\begin{tikzpicture}
\filldraw [black] (-5.5,0) circle (2pt);
\filldraw [black] (-4.5,0) circle (2pt);
\filldraw [black] (-3.5,0) circle (2pt);
\filldraw [black] (-2.5,0) circle (2pt);
\filldraw [black] (-1.5,0) circle (2pt);
\filldraw [black] (0,0) circle (2pt);
\filldraw [black] (1,0) circle (2pt);
\filldraw [black] (2,0) circle (2pt);
\draw  [black](0,0)--(1,0)--(2,0);
\draw [black] (1,0)--(1.75,0.75);
\draw [dashed] (0,0)--(-1.5,0);
\draw [black] (-5.5,0)--(-4.5,0)--(-3.5,0)--(-2.5,0)--(-1.5,0);
\draw [black] (-5.25,0.75)--(-4.5,0);
\filldraw [black] (-3.5,1) circle (2pt);
\filldraw [black] (1.75,0.75) circle (2pt);
\filldraw [black] (-5.25,0.75) circle (2pt);
\filldraw [black] (-2.5,1) circle (2pt);
\draw [black] (-3.5,0)--(-3.5,1);
\draw [black] (-2.5,0)--(-2.5,1);
\node at (-4.5,-0.35) {\small $v_1$};
\node at (-3.5,-0.35) {\small $v_2$};
\node at (-2.5,-0.35) {\small $v_3$};
\node at (1,-0.35) {\small $v_{n-6}$};
\node at (-1,-0.7) {$F_n$};
\end{tikzpicture}
\end{minipage}
\begin{minipage}{.495\linewidth}
\centering
\begin{tikzpicture}
\filldraw [black] (-4.5,0) circle (2pt);
\filldraw [black] (-3.5,0) circle (2pt);
\filldraw [black] (-2.5,0) circle (2pt);
\filldraw [black] (-1.5,0) circle (2pt);
\filldraw [black] (0,0) circle (2pt);
\filldraw [black] (1,0) circle (2pt);
\filldraw [black] (2,0) circle (2pt);
\draw  [black](0,0)--(1,0)--(2,0);
\draw [black] (1.75,-0.75)--(1,0)--(1.75,0.75);
\draw [dashed] (0,0)--(-1.5,0);
\draw [black] (-4.5,0)--(-3.5,0)--(-2.5,0)--(-1.5,0);
\draw [black] (-4.25,0.75)--(-3.5,0)--(-4.25,-0.75);
\filldraw [black] (-3.25,0.75) circle (2pt);
\filldraw [black] (-1.75,0.75) circle (2pt);
\filldraw [black] (1.75,0.75) circle (2pt);
\filldraw [black] (1.75,-0.75) circle (2pt);
\filldraw [black] (-4.25,0.75) circle (2pt);
\filldraw [black] (-4.25,-0.75) circle (2pt);
\draw [black] (-3.25,0.75)--(-2.5,0)--(-1.75,0.75);
\node at (-3.5,-0.35) {\small $v_1$};
\node at (-2.5,-0.35) {\small $v_2$};
\node at (0.9,-0.35) {\small $v_{n-8}$};
\node at (-1,-0.7) {$F_n^*$};
\end{tikzpicture}
\end{minipage}
\caption{The graphs $F_n$ and $F_n^*$ (left to right).}
\label{figfn}
\end{figure}
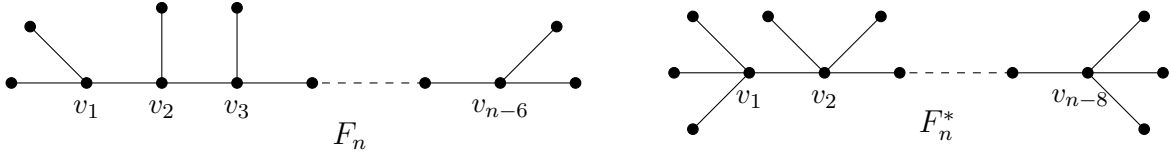 

\begin{proposition}\label{fn}
For $n\ge 10$ with $n\equiv 1\pmod 3$, $F_n$ is $1$-critical.
\end{proposition}
\begin{proof}
Let $u$ denote a pendant neighbor of $v_{n-6}$. Note that $F_n-uv_{n-6}\cong K_1\cup R_{n-1}$ and $F_n\setminus uv_{n-6}\cong K_1\cup R_{n-3}$.
By Proposition \ref{rule}, we have \[
\mu(F_n,x)=x\mu(R_{n-1},x)-x\mu(R_{n-3},x).
\]
As $n\equiv 1\pmod 3$, we have by Lemma \ref{rn} that
\[
\mu(F_n,1)=\mu(R_{n-1},1)-\mu(R_{n-3},1)=(-1)^{\frac{n-1}{3}+1}-(-1)^{\frac{n-1}{3}+1}=0.
\]
This shows that $1$ is a root of $\mu(F_n,x)$.

As $F_n\setminus u\cong R_{n-1}$ and $F_n\setminus v_{n-6}\cong 2K_1\cup R_{n-3}$, we have by Proposition~\ref{rule} and Lemma \ref{rn} that $\mu(F_n\setminus u,1)\ne 0$ and $\mu(F_n\setminus v_{n-6},1)\ne 0$. Note that $F_n\setminus v_2\cong K_1\cup Y_3\cup Y_{n-5}$, $F_n\setminus v_3\cong K_1\cup Y_5\cup Y_{n-7}$, $F_n\setminus v_i\cong R_{3+i}\cup Y_{n-4-i}$ with $4\le i\le n-8$ and $F_n\setminus v_{n-7}\cong R_{n-4}\cup Y_3$. Then by Lemmas \ref{y} and \ref{rn}, $\mu(F_n\setminus v_i,1)\ne 0$ for $i=2,\dots,n-7$. 

Let $u_i$ denote the pendant neighbor of $v_i$ in $F_n$ and $H_i=F_n\setminus u_i$ for $i=2,3$. As $H_i-v_{5-i}u_{5-i}\cong K_1\cup W_{n-2}$,  $H_2\setminus v_3u_3\cong Y_4\cup Y_{n-7}$ and $H_3\setminus v_2u_2\cong Y_3\cup Y_{n-5}$, we have by Eq.~\eqref{eqwnn} and Proposition \ref{rule} and Lemma \ref{y} that $\mu(H_i,1)\ne 0$.
Let $H_i=F_n\setminus v_i$ for $i=0,1$. As $H_0-v_2v_3\cong P_4\cup Y_{n-5}$, $H_0\setminus v_2v_3\cong 2K_1\cup K_2\cup Y_{n-7}$, $H_1-v_3u_3\cong 3K_1\cup Y_{n-4}$ and $H_1\setminus v_3u_3\cong 2K_1\cup K_2\cup Y_{n-7}$, it also follows from Proposition \ref{rule}, Lemma \ref{y} and the fact $\mu(K_2,1)=0$ and  $\mu(P_4,1)=-1$ that $\mu(H_i,1)=\mu(Y_{n-4},x)\ne 0$. This proves that each vertex of $F_n$ is $1$-essential. Therefore, by Theorem \ref{essential}, $F_n$ is $1$-critical.
\end{proof}

For $n\ge 4$, let $Y_n^*$ denote the graph obtained from the path $P_{n-2}:=v_0\dots v_{n-3}$ by adding two pendant vertices to $v_1$, see Fig. \ref{figyw}. Using induction and similar argument as in the proof of Lemma \ref{rn}, we derive the following result, whose proof is omitted.

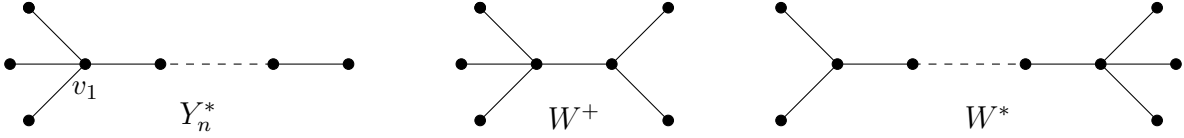
\begin{figure}[htbp]
\centering
\begin{tikzpicture}
\filldraw [black] (-9.5,0) circle (2pt);
\filldraw [black] (-8.5,0) circle (2pt);
\filldraw [black] (-7.5,0) circle (2pt);
\filldraw [black] (-6,0) circle (2pt);
\filldraw [black] (-5,0) circle (2pt);
\draw  [black](-6,0)--(-5,0);
\draw [dashed] (-6,0)--(-7.5,0);
\draw [black] (-9.5,0)--(-8.5,0)--(-7.5,0);
\draw [black] (-9.25,0.75)--(-8.5,0)--(-9.25,-0.75);
\filldraw [black] (-9.25,0.75) circle (2pt);
\filldraw [black] (-9.25,0.75) circle (2pt);
\filldraw [black] (-9.25,-0.75) circle (2pt);
\node at (-8.5,-0.35) {\small $v_1$};
\node at (-7,-0.7) {$Y_n^*$};
\filldraw [black] (-3.5,0) circle (2pt);
\filldraw [black] (-2.5,0) circle (2pt);
\filldraw [black] (-1.5,0) circle (2pt);
\draw [black] (-3.5,0)--(-2.5,0)--(-1.5,0)--(-0.75,0.75);
\draw [black] (-0.75,-0.75)--(-1.5,0);
\draw [black] (-3.25,0.75)--(-2.5,0)--(-3.25,-0.75);
\filldraw [black] (-3.25,0.75) circle (2pt);
\filldraw [black] (-3.25,0.75) circle (2pt);
\filldraw [black] (-3.25,-0.75) circle (2pt);
\filldraw [black] (-0.75,0.75) circle (2pt);
\filldraw [black] (-0.75,-0.75) circle (2pt);
\node at (-2,-0.7) {$W^+$};
\filldraw [black] (0.75,0.75) circle (2pt);
\filldraw [black] (0.75,-0.75) circle (2pt);
\draw [black] (0.75,0.75)--(1.5,0)--(0.75,-0.75);
\filldraw [black] (1.5,0) circle (2pt);
\filldraw [black] (2.5,0) circle (2pt);
\draw [black] (1.5,0)--(2.5,0);
\draw [dashed] (2.5,0)--(4,0);
\filldraw [black] (4,0) circle (2pt);
\filldraw [black] (5,0) circle (2pt);
\filldraw [black] (6,0) circle (2pt);
\draw [black] (4,0)--(5,0)-- (6,0);
\filldraw [black] (5.75,0.75) circle (2pt);
\filldraw [black] (5.75,-0.75) circle (2pt);
\draw [black] (5.75,0.75) --(5,0)--(5.75,-0.75);
\node at (3.5,-0.7) {$W^*$};
\end{tikzpicture}
\caption{The graphs $Y_n^*$, $W^+$ and $W^*$ (left to right).}
\label{figyw}
\end{figure}

\begin{lemma}\label{yn*}
For $n\ge 4$, $\mu(Y_n^*,1)=\begin{cases}
(-1)^{\frac{n-3}{3}}&\mbox{ if }n\equiv 0\pmod 3,\\
(-1)^{\frac{n-1}{3}}\cdot 2&\mbox{ if }n\equiv 1\pmod 3,\\
(-1)^{\frac{n-2}{3}}\cdot 3&\mbox{ if }n\equiv 2\pmod 3.
\end{cases}$
\end{lemma}

\begin{proposition}\label{fn*}
For $n\ge 11$ with $n\equiv 2\pmod 3$, $F_n^*$ is $1$-critical.
\end{proposition}
\begin{proof}
If $n=11$, then $\mu(F_n^*,x)=x^{11}-10x^9+27x^7-18x^5$ and so $1$ is a root of $\mu(F_n^*,x)$. As $F_n^*\setminus v_i\cong 3K_1\cup W^+$ for $i=1,3$,  where $W^+$ is shown in Fig. \ref{figyw}, we have by a direct calculation that $\mu(F_n^*\setminus v_i,1)=\mu(W^+,1)=1\ne 0$. As $F_n^*\setminus v_2\cong 2K_1\cup Y_4^*$, $\mu(F_n^*\setminus v_2,1)=\mu(Y_4^*,1)=-2\ne 0$ by Lemma \ref{yn*}. Let $u$ denote a pendant neighbor of $v_2$. Let $G_1=F_n^*\setminus v_0$ and $G_2=F_n^*\setminus u$. Note that $G_1-v_1v_2\cong Y_3\cup W^+$, $G_1\setminus v_1v_2\cong 4K_1\cup Y_4^*$, $G_2-v_1v_2\cong Y_4^*\cup Y_6^*$ and $G_2\setminus v_1v_2\cong 4K_1\cup Y_4^*$. Then, by Proposition \ref{rule}, Lemmas \ref{y} and \ref{yn*} and the fact that $\mu(W^+,1)=1$, we have \[
\mu(G_1,1)=\mu(Y_3,1)\mu(W^+,1)-\mu(Y_4^*,1)=(-1)-(-2)=1\ne 0
\]
and \[
\mu(G_2,1)=\mu(Y_4^*,1)\mu(Y_6^*,1)-\mu(Y_4^*,1)=2-(-1)=3\ne 0.
\]
For any pendant vertex $w$ in $F_n^*$, if $wv_2\in E(G)$, then $F_n^*\setminus w\cong G_2$, and otherwise, $F_n^*\setminus w\cong G_1$. So every vertex in $F_n^*$ is $1$-essential and therefore, $F_n^*$ is $1$-critical by Theorem \ref{essential}.

Assume in the following that $n\ge 14$.
Note that $F_n^*-v_{n-9}v_{n-8}\cong Y_{4}\cup R_{n-4}^*$ and $F_n^*\setminus v_{n-9}v_{n-8}\cong 3K_1\cup R_{n-5}^*$. By Proposition \ref{rule}, we have
\[
\mu(F_n^*,x)=\mu(Y_{4},x)\mu(R_{n-4}^*,x)-x^3\mu(R_{n-5}^*,x).
\]
As $n\equiv 2\pmod 3$, it follows by Lemmas \ref{y} and Lemma \ref{rn} that \[
\mu(F_n^*,1)=-2\cdot (-1)^{\frac{n-5}{3}}-(-1)^{\frac{n-8}{2}}\cdot 2=0,
\]
i.e., $1$ is a root of $\mu(F_n^*,x)$.

As $F_n^*\setminus v_{n-8}\cong 3K_1\cup R_{n-4}^*$, $F_n^*\setminus v_i\cong R_{i-1}^*\cup Y_{n-i}^*$ for $i=4,\dots,n-9$, $F_n^*\setminus v_3\cong W^+\cup Y_{n-8}^*$ and $F_n^*\setminus v_2\cong 2K_1\cup Y_4^*\cup Y_{n-7}^*$, we have by Proposition \ref{rule}, Lemmas \ref{rn} and \ref{yn*} and the fact that $\mu(W^+,1)=1$ that $\mu(F_n^*\setminus v_i,1)\ne 0$ for $i=2,\dots,n-8$. 

Note that $F_n^*\setminus v_1\cong 3K_1\cup W^*$, where $W^*$ is a graph of order $n-4$ shown in Fig. \ref{figyw}. Let $u$ be the vertex of degree $4$ in $W^*$ and $u'$ a pendant neighbor of $u$. As $n\equiv 2\pmod 3$, $W^*-uu'\cong K_1\cup W_{n-5}$ and $W^*\setminus uu'\cong 2K_1\cup Y_{n-8}$, we have by Proposition \ref{rule}, Eq.~\eqref{eqwnn} and Lemmas \ref{y} that
\begin{equation}\label{eqw}
\mu(W^*,1)=\mu(W_{n-5},1)-\mu(Y_{n-8},1)=(-1)^{\frac{n-2}{3}}\ne 0
\end{equation} 
and so $\mu(F_n^*\setminus v_1,1)\ne 0$. 

Denote by $v$ a pendant neighbor of $v_2$. Let $H_1=F_n^*\setminus v_0$, $H_2=F_n^*\setminus v$ and $H_3=F_n^*\setminus v_{n-7}$.
Note that $H_1-v_1v_2\cong Y_3\cup W^*$, $H_1\setminus v_1v_2\cong 4K_1\cup Y^*_{n-7}$, $H_2-v_1v_2\cong Y^*_4\cup Y^*_{n-5}$, $H_2\setminus v_1v_2\cong 4K_1\cup Y^*_{n-7}$, $H_3-v_1v_2\cong Y^*_4\cup W_{n-5}$ and $H_3\setminus v_1v_2\cong 5K_1\cup Y_{n-8}$. Then, by Proposition \ref{rule}, Lemmas \ref{y} and \ref{yn*}, and Eq.~\eqref{eqw} and \ref{eqwnn}, we have 
\[
\mu(H_1,1)=\mu(Y_3,1)\mu(W^*,1)-\mu(Y^*_{n-7},1)=(-1)\cdot(-1)^{\frac{n-2}{3}}-(-1)^{\frac{n-8}{3}}\cdot 2=(-1)^{\frac{n-2}{3}}\cdot 3\ne 0,
\]
\[
\mu(H_2,1)=\mu(Y^*_4,1)\mu(Y^*_{n-5},1)-\mu(Y^*_{n-7},1)=(-2)\cdot (-1)^{\frac{n-8}{3}}-(-1)^{\frac{n-8}{3}}\cdot 2=(-1)^{\frac{n+1}{3}}\cdot 4\ne 0
\]
and \[
\mu(H_3,1)=\mu(Y^*_4,1)\mu(W_{n-5},1)-\mu(Y_{n-8},1)=(-1)^{\frac{n-8}{3}}\ne 0.
\]
So every pendant vertex in $F_n^*$ is $1$-essential. Therefore, every vertex in $F_n^*$ is $1$-essential. By Theorem \ref{essential}, $F_n^*$ is $1$-critical.
\end{proof} 

It's easy to see that $K_2$ is the $1$-critical graph of minimum order. Evidently, it is also the $1$-critical tree of minimum order. It is natural to ask whether, for each order less than 9, there exists a $1$-critical tree. By Proposition \ref{wn}, $W_6$ is a $1$-critical tree of order $6$. However, it can be shown by SageMath, see Appendix A, that there is no $1$-critical trees of order $n$ with $n=3,4,5,7,8$.
However, it can be shown using SageMath (see Appendix A) that there is no $1$-critical tree of order $3,4,5,7$ or $8$.

\subsection{Existence of $1$-critical graphs}

The second part of Theorem \ref{1} is proved in this subsection by constructing $1$-critical unicyclic graphs for $n\ge 5$ except $n=7$ (see Proposition \ref{unic}), and by confirming the case $n=7$ with SageMath (see Appendix A). 


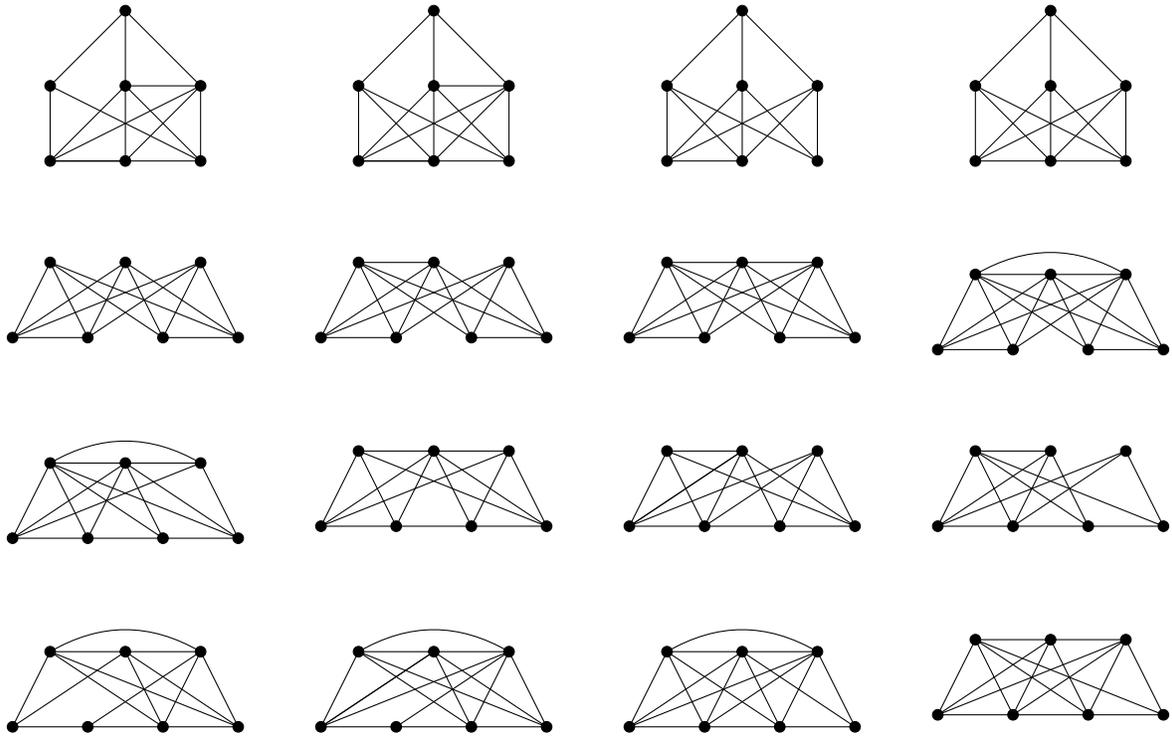
\begin{figure}[htbp]
\centering
\begin{minipage}{.24\textwidth}
\centering
\begin{tikzpicture}
\filldraw [black] (0,0) circle (2pt);
\filldraw [black] (1,0) circle (2pt);
\filldraw [black] (0,1) circle (2pt);
\filldraw [black] (-1,0) circle (2pt);
\filldraw [black] (0,-1) circle (2pt);
\filldraw [black] (-1,-1) circle (2pt);
\filldraw [black] (1,-1) circle (2pt);
\draw [black] (0,0)--(1,-1)--(1,0)--(0,-1)--(-1,-1)--(0,0)--(0,-1)--(-1,-1)--(-1,0)--(0,1)--(0,0)--(1,0)--(0,1);
\draw[black] (0,-1)--(1,-1)--(-1,0);
\draw[black] (1,0)--(-1,-1);
\end{tikzpicture}
\end{minipage}
\begin{minipage}{.24\textwidth}
\centering
\begin{tikzpicture}
\filldraw [black] (0,0) circle (2pt);
\filldraw [black] (1,0) circle (2pt);
\filldraw [black] (0,1) circle (2pt);
\filldraw [black] (-1,0) circle (2pt);
\filldraw [black] (0,-1) circle (2pt);
\filldraw [black] (-1,-1) circle (2pt);
\filldraw [black] (1,-1) circle (2pt);
\draw [black] (0,0)--(1,-1)--(1,0)--(0,-1)--(-1,-1)--(0,0)--(0,-1)--(-1,-1)--(-1,0)--(0,1)--(0,0)--(1,0)--(0,1);
\draw[black] (0,-1)--(1,-1)--(-1,0)--(0,-1);
\draw[black] (1,0)--(-1,-1);
\end{tikzpicture}
\end{minipage}
\begin{minipage}{.24\textwidth}
\centering
\begin{tikzpicture}
\filldraw [black] (0,0) circle (2pt);
\filldraw [black] (1,0) circle (2pt);
\filldraw [black] (0,1) circle (2pt);
\filldraw [black] (-1,0) circle (2pt);
\filldraw [black] (0,-1) circle (2pt);
\filldraw [black] (-1,-1) circle (2pt);
\filldraw [black] (1,-1) circle (2pt);
\draw [black] (0,0)--(1,-1)--(1,0)--(0,-1)--(-1,-1)--(0,0)--(0,-1);
\draw [black] (-1,-1)--(-1,0)--(0,1)--(0,0);
\draw[black] (1,-1)--(-1,0)--(0,-1);
\draw[black] (0,1)--(1,0)--(-1,-1);
\end{tikzpicture}
\end{minipage}
\begin{minipage}{.24\textwidth}
\centering
\begin{tikzpicture}
\filldraw [black] (0,0) circle (2pt);
\filldraw [black] (1,0) circle (2pt);
\filldraw [black] (0,1) circle (2pt);
\filldraw [black] (-1,0) circle (2pt);
\filldraw [black] (0,-1) circle (2pt);
\filldraw [black] (-1,-1) circle (2pt);
\filldraw [black] (1,-1) circle (2pt);
\draw [black] (0,0)--(1,-1)--(1,0)--(0,-1)--(-1,-1)--(0,0)--(0,-1);
\draw [black] (-1,-1)--(-1,0)--(0,1)--(0,0);
\draw[black] (0,-1)--(1,-1)--(-1,0)--(0,-1);
\draw[black] (0,1)--(1,0)--(-1,-1);
\end{tikzpicture}
\end{minipage}
\\[1cm]
\begin{minipage}{.24\textwidth}
\centering
\begin{tikzpicture}
\filldraw [black] (-1,1) circle (2pt);
\filldraw [black] (1,1) circle (2pt);
\filldraw [black] (0,1) circle (2pt);
\filldraw [black] (-0.5,0) circle (2pt);
\filldraw [black] (0.5,0) circle (2pt);
\filldraw [black] (1.5,0) circle (2pt);
\filldraw [black] (-1.5,0) circle (2pt);
\draw [black] (0,1)--(1.5,0)--(0.5,0)--(0,1)--(-1.5,0)--(-0.5,0)--(0,1);
\draw [black] (1.5,0)--(-1,1)--(-1.5,0);
\draw [black] (-0.5,0)--(-1,1)--(0.5,0);
\draw[black] (-1.5,0)--(1,1)--(1.5,0);
\draw[black] (-0.5,0)--(1,1)--(0.5,0);
\end{tikzpicture}
\end{minipage}
\begin{minipage}{.24\textwidth}
\centering
\begin{tikzpicture}
\filldraw [black] (-1,1) circle (2pt);
\filldraw [black] (1,1) circle (2pt);
\filldraw [black] (0,1) circle (2pt);
\filldraw [black] (-0.5,0) circle (2pt);
\filldraw [black] (0.5,0) circle (2pt);
\filldraw [black] (1.5,0) circle (2pt);
\filldraw [black] (-1.5,0) circle (2pt);
\draw [black] (0,1)--(1.5,0)--(0.5,0)--(0,1)--(-1.5,0)--(-0.5,0)--(0,1)--(-1,1)--(-1.5,0);
\draw [black] (-0.5,0)--(-1,1)--(0.5,0);
\draw[black] (-1.5,0)--(1,1)--(1.5,0)--(-1,1);
\draw[black] (-0.5,0)--(1,1)--(0.5,0);
\end{tikzpicture}
\end{minipage}
\begin{minipage}{.24\textwidth}
\centering
\begin{tikzpicture}
\filldraw [black] (-1,1) circle (2pt);
\filldraw [black] (1,1) circle (2pt);
\filldraw [black] (0,1) circle (2pt);
\filldraw [black] (-0.5,0) circle (2pt);
\filldraw [black] (0.5,0) circle (2pt);
\filldraw [black] (1.5,0) circle (2pt);
\filldraw [black] (-1.5,0) circle (2pt);
\draw [black] (1,1)--(0,1)--(1.5,0)--(0.5,0)--(0,1)--(-1.5,0)--(-0.5,0)--(0,1)--(-1,1)--(-1.5,0);
\draw [black] (-0.5,0)--(-1,1)--(0.5,0);
\draw[black] (-1.5,0)--(1,1)--(1.5,0)--(-1,1);
\draw[black] (-0.5,0)--(1,1)--(0.5,0);
\end{tikzpicture}
\end{minipage}
\begin{minipage}{.24\textwidth}
\centering
\begin{tikzpicture}
\filldraw [black] (-1,1) circle (2pt);
\filldraw [black] (1,1) circle (2pt);
\filldraw [black] (0,1) circle (2pt);
\filldraw [black] (-0.5,0) circle (2pt);
\filldraw [black] (0.5,0) circle (2pt);
\filldraw [black] (1.5,0) circle (2pt);
\filldraw [black] (-1.5,0) circle (2pt);
\draw [black] (1,1)--(0,1)--(1.5,0)--(0.5,0)--(0,1)--(-1.5,0)--(-0.5,0)--(0,1)--(-1,1)--(-1.5,0);
\draw [black] (-0.5,0)--(-1,1)--(0.5,0);
\draw[black] (-1.5,0)--(1,1)--(1.5,0)--(-1,1);
\draw[black] (-0.5,0)--(1,1)--(0.5,0);
\draw[black] (-1,1) to[out=30, in=150] (1,1);
\end{tikzpicture}
\end{minipage}
\\[1cm]
\begin{minipage}{.24\textwidth}
\centering
\begin{tikzpicture}
\filldraw [black] (-1,1) circle (2pt);
\filldraw [black] (1,1) circle (2pt);
\filldraw [black] (0,1) circle (2pt);
\filldraw [black] (-0.5,0) circle (2pt);
\filldraw [black] (0.5,0) circle (2pt);
\filldraw [black] (1.5,0) circle (2pt);
\filldraw [black] (-1.5,0) circle (2pt);
\draw [black] (1,1)--(0,1)--(-1,1);
\draw [black] (-1.5,0)--(-0.5,0)--(0.5,0)--(1.5,0);
\draw[black] (-1,1)--(-0.5,0)--(0,1)--(0.5,0)--(-1,1);
\draw [black] (1.5,0)--(1,1)--(-1.5,0)--(0,1)--(1.5,0)--(-1,1)--(-1.5,0);
\draw[black] (-1,1) to[out=30, in=150] (1,1);
\end{tikzpicture}
\end{minipage}
\begin{minipage}{.24\textwidth}
\centering
\begin{tikzpicture}
\filldraw [black] (-1,1) circle (2pt);
\filldraw [black] (1,1) circle (2pt);
\filldraw [black] (0,1) circle (2pt);
\filldraw [black] (-0.5,0) circle (2pt);
\filldraw [black] (0.5,0) circle (2pt);
\filldraw [black] (1.5,0) circle (2pt);
\filldraw [black] (-1.5,0) circle (2pt);
\draw [black] (1,1)--(0,1)--(-1,1);
\draw [black] (0,1)--(-1.5,0)--(-0.5,0)--(0.5,0)--(1.5,0);
\draw[black] (1,1)--(-1.5,0)--(-1,1)--(-0.5,0)--(0,1)--(1.5,0)--(-1,1);
\draw [black] (0,1)--(0.5,0)--(1,1)--(1.5,0);
\end{tikzpicture}
\end{minipage}
\begin{minipage}{.24\textwidth}
\centering
\begin{tikzpicture}
\filldraw [black] (-1,1) circle (2pt);
\filldraw [black] (1,1) circle (2pt);
\filldraw [black] (0,1) circle (2pt);
\filldraw [black] (-0.5,0) circle (2pt);
\filldraw [black] (0.5,0) circle (2pt);
\filldraw [black] (1.5,0) circle (2pt);
\filldraw [black] (-1.5,0) circle (2pt);
\draw [black] (0,1)--(-1,1);
\draw [black] (0,1)--(-1.5,0)--(-0.5,0)--(0.5,0)--(1.5,0)--(-1,1);
\draw[black] (-1.5,0)--(1,1)--(1.5,0)--(0,1)--(0.5,0)--(1,1);
\draw[black] (-0.5,0)--(-1,1)--(-1.5,0)--(0,1)--(-0.5,0)--(1,1);
\end{tikzpicture}
\end{minipage}
\begin{minipage}{.24\textwidth}
\centering
\begin{tikzpicture}
\filldraw [black] (-1,1) circle (2pt);
\filldraw [black] (1,1) circle (2pt);
\filldraw [black] (0,1) circle (2pt);
\filldraw [black] (-0.5,0) circle (2pt);
\filldraw [black] (0.5,0) circle (2pt);
\filldraw [black] (1.5,0) circle (2pt);
\filldraw [black] (-1.5,0) circle (2pt);
\draw [black] (0,1)--(-1,1);
\draw [black] (-1,1)--(-1.5,0)--(-0.5,0)--(0.5,0)--(1.5,0)--(1,1);
\draw [black] (0,1)--(-1.5,0)--(1,1)--(-0.5,0)--(-1,1)--(0.5,0)--(0,1)--(-0.5,0);
\draw [black] (-1,1)--(1.5,0);
\end{tikzpicture}
\end{minipage}
\\[1cm]
\begin{minipage}{.24\textwidth}
\centering
\begin{tikzpicture}
\filldraw [black] (-1,1) circle (2pt);
\filldraw [black] (1,1) circle (2pt);
\filldraw [black] (0,1) circle (2pt);
\filldraw [black] (-0.5,0) circle (2pt);
\filldraw [black] (0.5,0) circle (2pt);
\filldraw [black] (1.5,0) circle (2pt);
\filldraw [black] (-1.5,0) circle (2pt);
\draw [black] (1,1)--(0,1)--(-1,1);
\draw [black] (0.5,0)--(-1,1)--(-1.5,0)--(-0.5,0)--(0.5,0)--(1.5,0)--(1,1);
\draw[black] (-1,1) to[out=30, in=150] (1,1);
\draw [black] (-1.5,0)--(0,1)--(1.5,0)--(-1,1);
\draw [black] (-0.5,0)--(1,1)--(0.5,0)--(0,1);
\end{tikzpicture}
\end{minipage}
\begin{minipage}{.24\textwidth}
\centering
\begin{tikzpicture}
\filldraw [black] (-1,1) circle (2pt);
\filldraw [black] (1,1) circle (2pt);
\filldraw [black] (0,1) circle (2pt);
\filldraw [black] (-0.5,0) circle (2pt);
\filldraw [black] (0.5,0) circle (2pt);
\filldraw [black] (1.5,0) circle (2pt);
\filldraw [black] (-1.5,0) circle (2pt);
\draw [black] (1,1)--(0,1)--(-1,1);
\draw [black] (-1,1)--(-1.5,0)--(-0.5,0)--(0.5,0)--(1.5,0)--(1,1);
\draw[black] (-1,1) to[out=30, in=150] (1,1);
\draw [black] (-1.5,0)--(0,1)--(1.5,0)--(-1,1)--(0.5,0)--(1,1)--(-0.5,0);
\draw [black] (1,1)--(-1.5,0)--(0,1)--(0.5,0);
\end{tikzpicture}
\end{minipage}
\begin{minipage}{.24\textwidth}
\centering
\begin{tikzpicture}
\filldraw [black] (-1,1) circle (2pt);
\filldraw [black] (1,1) circle (2pt);
\filldraw [black] (0,1) circle (2pt);
\filldraw [black] (-0.5,0) circle (2pt);
\filldraw [black] (0.5,0) circle (2pt);
\filldraw [black] (1.5,0) circle (2pt);
\filldraw [black] (-1.5,0) circle (2pt);
\draw [black] (1,1)--(0,1)--(-1,1);
\draw [black] (-1,1)--(-1.5,0)--(-0.5,0)--(0.5,0)--(1.5,0)--(1,1);
\draw[black] (-1,1) to[out=30, in=150] (1,1);
\draw [black] (1.5,0)--(0,1)--(-1.5,0)--(1,1)--(0.5,0)--(-1,1)--(-0.5,0)--(1,1);
\draw [black] (-0.5,0)--(0,1)--(0.5,0);
\end{tikzpicture}
\end{minipage}
\begin{minipage}{.24\textwidth}
\centering
\begin{tikzpicture}
\filldraw [black] (-1,1) circle (2pt);
\filldraw [black] (1,1) circle (2pt);
\filldraw [black] (0,1) circle (2pt);
\filldraw [black] (-0.5,0) circle (2pt);
\filldraw [black] (0.5,0) circle (2pt);
\filldraw [black] (1.5,0) circle (2pt);
\filldraw [black] (-1.5,0) circle (2pt);
\draw [black] (1,1)--(0,1)--(-1,1);
\draw [black] (-1,1)--(-1.5,0)--(-0.5,0)--(0.5,0)--(1.5,0)--(1,1);
\draw [black] (1,1)--(-1.5,0)--(0,1)--(-0.5,0)--(1,1)--(0.5,0)--(-1,1)--(-0.5,0);
\draw [black] (1.5,0)--(-1,1);
\draw [black] (0,1)--(0.5,0);
\end{tikzpicture}
\end{minipage}
\caption{$1$-critical graphs on $7$ vertices.}
\label{critical7}
\end{figure}

For $n\ge 6$, let $C_n^*$ denote the graph obtained from the cycle $C_{n-3}:=v_0\dots v_{n-4}v_0$ by attaching two pendant vertices to $v_0$ and one pendant vertex to $v_2$, see Fig.~\ref{figc}. For $n\ge 10$, let $\hat{C}_n$ denote the graph obtained from the cycle $C_{n-5}:=v_0\dots v_{n-6}v_0$ by attaching one pendant vertex to each vertex of $\{v_0,v_3,v_4\}$ and two pendant vertices to $v_2$, see Fig. \ref{figc}. For $n\ge 5$, let $C_n^+$ denote the graph obtained from the cycle $C_{n-1}$ by attaching a pendant vertex to one vertex of $C_{n-1}$, also see Fig.~\ref{figc}.

\begin{figure}[htbp]
\centering
\begin{tikzpicture}
\draw [black](-1,0) circle (1.2);
\filldraw [black] (-0.25,1.75) circle (2pt);
\filldraw [black] (-1.75,1.75) circle (2pt);
\filldraw [black] (-1,1.2) circle (2pt);
\filldraw [black] (-0.15,0.82) circle (2pt);
\filldraw [black] (0.2,0) circle (2pt);
\filldraw [black] (1.2,0) circle (2pt);
\draw [black] (-0.25,1.75)--(-1,1.2)--(-1.75,1.75);
\draw [black] (1.2,0)--(0.2,0);
\node at (-1,0.9) {\small $v_0$};
\node at (0.15,0.82) {\small $v_1$};
\node at (0.4,-0.3) {\small $v_2$};
\node at (-1,0) {\small $C_{n-3}$};
\node at (0.5,-1.25) {$C_n^*$};
\draw [black](5,0) circle (1.2);
\filldraw [black] (5.75,1.75) circle (2pt);
\filldraw [black] (4.25,1.75) circle (2pt);
\filldraw [black] (5,1.2) circle (2pt);
\filldraw [black] (5.85,0.82) circle (2pt);
\filldraw [black] (4.15,0.82) circle (2pt);
\filldraw [black] (2.8,0) circle (2pt);
\filldraw [black] (3.8,0) circle (2pt);
\filldraw [black] (6.2,0) circle (2pt);
\filldraw [black] (7.2,0) circle (2pt);
\filldraw [black] (6.85,0.82) circle (2pt);
\draw [black]  (5.85,0.82)--(6.85,0.82);
\draw [black] (5.75,1.75)--(5,1.2)--(4.25,1.75);
\draw [black] (6.2,0)--(7.2,0);
\draw [black] (2.8,0)--(3.8,0);
\node at (3.5,-0.3) {\small $v_0$};
\node at (3.85,0.82){\small $v_1$};
\node at (5,0.9){\small $v_2$};
\node at (6.15,0.95) {\small $v_3$};
\node at (6.4,-0.3){\small $v_4$};
\node at (5,0) {\small $C_{n-5}$};
\node at (6.5,-1.25) {$\hat{C}_n$};
\draw [black](10,0) circle (1.2);
\filldraw [black] (11.2,0) circle (2pt);
\filldraw [black] (12.2,0) circle (2pt);
\draw [black] (11.2,0)--(12.2,0);
\node at (10,0) {\small $C_{n-1}$};
\node at (11.5,-1.25) {$C_n^+$};
\end{tikzpicture}
\caption{The graph $C_n^*$, $\hat{C}_n$ and $C_n^+$ (left to right).}
\label{figc}
\end{figure}
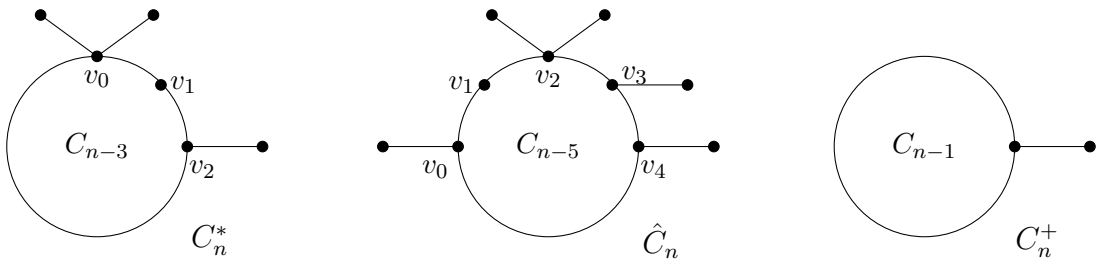

\begin{proposition}\label{unic}
(i) For $n\ge 6$ with $n\equiv 0\pmod 3$, $C_n^*$ is $1$-critical;\\
(ii) For $n\ge 10$ with $n\equiv 1\pmod 3$, $\hat{C}_n$ is $1$-critical;\\
(iii) For $n\ge 5$ with $n\equiv 2\pmod 3$, $C_n^+$ is $1$-critical.
\end{proposition}

The proof is almost similar to Propositions \ref{wn}--\ref{fn*}, and is omitted.

%
%

It's worth noting that there are exactly $16$ graphs of order $7$ that is $1$-critical (see Fig. \ref{critical7}), and all of them are not unicyclic.
In Fig.~\ref{critical7}, one can be observed that some graphs remain 1-critical after adding an edge. For example, adding an edge to the last graph in the third row yields the last graph in the fourth row, both of them are $1$-critical.
It's natural to ask when a $1$-critical graph remains $1$-critical after edge addition.
So in the final part, we provide a construction that generates a new $1$-critical graph from an existing one.

\begin{theorem}\label{addedges}
Let $G$ be a $1$-critical graph with a cut vertex $u$ of degree three. If $G\setminus u$ consists of three components, of which at least one is trivial, then $G+u_1u_2$ is $1$-critical, where $u_1$ and $u_2$ are neighbors of $u$ such that the remaining neighbor is pendant.
\end{theorem}
\begin{proof}
Let $G'=G+u_1u_2$. As $G'\setminus u_1u_2=G\setminus u_1u_2$ and $G\setminus u_1u_2$ contains $K_2$ as a component, we have by Proposition \ref{rule} that $\mu(G'\setminus u_1u_2,x)$ has $1$ as a root and $\mu(G',x)=\mu(G,x)-\mu(G\setminus u_1u_2,x)$, so $1$ is a root of $\mu(G',x)$.

Let $z\in V(G)\setminus N[u]$. Let $T$ denote the component of $G\setminus z$ containing $u$ and $T_1,\dots,T_s$ the remaining components of $G\setminus z$ if $s:=d(z)-1\ge 1$. Then by Proposition \ref{rule} again, $\mu(T,x)$, $\mu(T_1,x), \dots,\mu(T_s,x)$ do not contain $1$ as a root.
Note that $T+vw\setminus u_1u_2$ coincides with $T\setminus u_1u_2$, and it  contains $K_2$ as a component. Moreover, \[
\mu(T+u_1u_2,x)=\mu(T,x)-\mu(T\setminus u_1u_2,x).
\]
Then $\mu(T+u_1u_2,x)$ does not contain $1$ as a root. 
As $G'\setminus z\cong T+vw$ if $s=0$ and $G'\setminus z\cong T+u_1u_2\cup T_1\cup \dots\cup T_s$ if $s\ge 1$, $1$ is not a root of $\mu(G'\setminus z,x)$.

Let $u_0$ denote the neighbor of $u$ different from $u_1$ and $u_2$. 
By Proposition \ref{rule}, 
\[
\mu(G,x)=x\mu(G\setminus u_0,x)-\mu(G\setminus uu_0,x),
\]
\[
\mu(G,x)=x\mu(G\setminus u,x)-\mu(G\setminus uu_0,x)-\mu(G\setminus uu_1,x)-\mu(G\setminus uu_2,x)
\]
and $\mu(G\setminus u,x)=x\mu(G\setminus uu_0,x)$.
Then, as $G$ is $1$-critical, 
\[
\mu(G\setminus u_0,1)=\mu(G\setminus u,1) 
\]
and 
\[
\mu(G\setminus uu_1,1)=-\mu(G\setminus uu_2,1).
\]
Let $F_i$ denote the component of $G\setminus u$ containing $u_i$ for $i=1,2$. We have $G\setminus u\cong K_1\cup F_1\cup F_2$ and $\mu(F_i,1)\ne 0$ for $i=1,2$. Moreover, as $G\setminus uu_1=K_1\cup (F_1\setminus u_1)\cup F_2$ and $G\setminus uu_2=K_1\cup F_1\cup (F_2\setminus u_2)$,
\begin{equation}\label{eqf12}
\mu(F_1\setminus u_1,1)\mu(F_2,1)=-\mu(F_1,1)\mu(F_2\setminus u_2,1).
\end{equation}
Note that $G'\setminus u_0u_1u_2\cong G\setminus u_0u_1u_2\cong (F_1\setminus u_1)\cup (F_2\setminus u_2)$. We have by Proposition \ref{rule} again that 
\[
\mu(G'\setminus u_0,1)=\mu(G\setminus u_0,1)-\mu(G'\setminus u_0u_1u_2,1)=\mu(G\setminus u,1)-\mu(G'\setminus uu_1u_2,1)
\]
and
\[
\mu(G'\setminus u,1)=\mu(G\setminus u,1)-\mu(G'\setminus uu_1u_2,1)=\mu(G\setminus u,1)-\mu(G'\setminus uu_1u_2,1).
\]
So, by Eq.~\eqref{eqf12},
\begin{align*}
\mu(G'\setminus u_0,1)=\mu(G'\setminus u,1)&=\mu(G\setminus u,1)-\mu(G'\setminus uu_1u_2,1)\\
&=\mu(F_1,1)\mu(F_2,1)-\mu(F_1\setminus u_1,1)\mu(F_2\setminus u_2,1)\\
&=\mu(F_1,1)\mu(F_2,1)+\frac{\mu(F_1,1)\mu(F_2\setminus u_2,1)^2}{\mu(F_2,1)}\\
&=\frac{\mu(F_1,1)\left(\mu(F_2,1)^2+\mu(F_2\setminus u_2,1)^2\right)}{\mu(F_2,1)}\\
&>0.
\end{align*}
This shows that $1$ is not a root of $G'\setminus u_0$ and $G'\setminus u$. Therefore, $G'$ is $1$-critical.
\end{proof}

\noindent
{\bf Example.}
Let $W_n^+$ and $F_n^+$ be the graph obtained from $W_n$ and $F_n$ by adding an edge (the red edge shown in Fig. \ref{figadd}), respectively.
By Theorem \ref{addedges} and Propositions \ref{wn}--\ref{fn}, $W_n^+$ and $F_n^+$ are $1$-critical.

\begin{figure}[htbp]
\centering
\begin{minipage}{.495\linewidth}
\centering
\begin{tikzpicture}
\filldraw [black] (-3.5,0) circle (2pt);
\filldraw [black] (-2.5,0) circle (2pt);
\filldraw [black] (-1.5,0) circle (2pt);
\filldraw [black] (0,0) circle (2pt);
\filldraw [black] (1,0) circle (2pt);
\filldraw [black] (2,0) circle (2pt);
\draw  [black](0,0)--(1,0)--(2,0);
\draw [black] (1,0)--(1.75,0.75);
\draw [dashed] (0,0)--(-1.5,0);
\draw [black] (-3.5,0)--(-2.5,0)--(-1.5,0);
\draw [black] (-3.25,0.75)--(-2.5,0);
\filldraw [black] (1.75,0.75) circle (2pt);
\filldraw [black] (-3.25,0.75) circle (2pt);
\draw[red]  (-3.25,0.75) to[out=30, in=135]  (-1.5,0);
\node at (-0.75,-0.5) {$W_n^+$};
\end{tikzpicture}
\end{minipage}
\begin{minipage}{.495\linewidth}
\centering
\begin{tikzpicture}
\filldraw [black] (-5.5,0) circle (2pt);
\filldraw [black] (-4.5,0) circle (2pt);
\filldraw [black] (-3.5,0) circle (2pt);
\filldraw [black] (-2.5,0) circle (2pt);
\filldraw [black] (-1.5,0) circle (2pt);
\filldraw [black] (0,0) circle (2pt);
\filldraw [black] (1,0) circle (2pt);
\filldraw [black] (2,0) circle (2pt);
\draw  [black](0,0)--(1,0)--(2,0);
\draw [black] (1,0)--(1.75,0.75);
\draw [dashed] (0,0)--(-1.5,0);
\draw [black] (-5.5,0)--(-4.5,0)--(-3.5,0)--(-2.5,0)--(-1.5,0);
\draw [black] (-5.25,0.75)--(-4.5,0);
\filldraw [black] (-3.5,1) circle (2pt);
\filldraw [black] (1.75,0.75) circle (2pt);
\filldraw [black] (-5.25,0.75) circle (2pt);
\filldraw [black] (-2.5,1) circle (2pt);
\draw [black] (-3.5,0)--(-3.5,1);
\draw [black] (-2.5,0)--(-2.5,1);
\node at (-4.5,-0.35) {\small $v_1$};
\node at (-3.5,-0.35) {\small $v_2$};
\node at (-2.5,-0.35) {\small $v_3$};
\node at (1,-0.35) {\small $v_{n-6}$};
\draw[red]  (-5.25,0.75) to[out=30, in=135]  (-3.5,0);
\node at (-1,-0.7) {$F_n^+$};
\end{tikzpicture}
\end{minipage}
\caption{The graphs $W_n^+$ and $F_n^+$ (left to right).}
\label{figadd}
\end{figure}
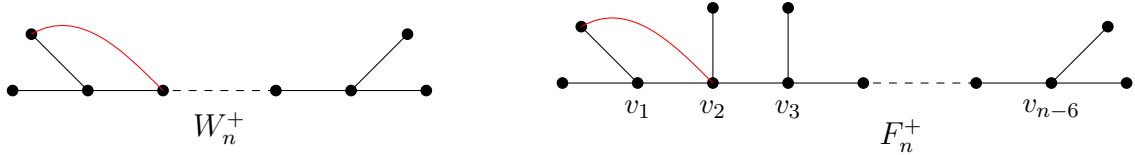

\section*{Appendix A}

Firstly, we define a function in SageMath to determine whether a given graph with $1$ as a matching polynomial root is $1$-critical.

\begin{lstlisting}
def delete(g):
    for v in g.vertices():
        h = g.copy()
        h.delete_vertex(v)
        hm = h.matching_polynomial()
        hresult = hm.subs(x=1)
        if hresult == 0:
            return False
    return True
\end{lstlisting}

The code below is used to show there is no $1$-critical trees of order $n$, where $n=3,4,5,7$.

\begin{lstlisting}
from sage.graphs.graph import Graph

n=8  # 3,4,5,7 
graphSet = []
for g in graphs.trees(n):
    m_p = g.matching_polynomial()
    result = m_p.subs(x=1)
    if result == 0:
        if delete(g):
            graphSet.append(g)
            
print("resulting graph:", len(graphSet))
\end{lstlisting}

The code below is used to find out all connected $1$-critical graphs of order $7$.

\begin{lstlisting}
from sage.graphs.graph import Graph

n = 7
graphSet = []
for g in graphs.nauty_geng(n):
    m_p = g.matching_polynomial()
    result = m_p.subs(x=1)
    if result == 0:
        if delete(g):
            graphSet.append(g)

print("resulting graph:", len(graphSet))
for g in graphSet:
    g.show()
    print("matching_polynomial:", g.matching_polynomial())
\end{lstlisting}

\end{document}